\documentclass[12pt,reqno]{article}

\usepackage[english]{babel}
\usepackage{graphicx}
\usepackage{amsthm}
\usepackage{amsmath,amsthm,amssymb}
\usepackage{times}
\usepackage{enumerate}
\usepackage{amsmath}
\usepackage{amsfonts}
\usepackage{pdfsync}
\usepackage{cleveref}
\usepackage{autonum}
\usepackage[usenames, dvipsnames]{color}
\usepackage{ulem} 

\newcommand{\R}{{\mathbb R}}
\newcommand{\N}{{\mathbb N}}

\newtheorem{theorem}{Theorem}
\newtheorem{corollary}[theorem]{Corollary}
\newtheorem{lemma}[theorem]{Lemma}

\newtheorem{proposition}[theorem]{Proposition}

\newtheorem{remark}[theorem]{Remark}
\newtheorem{example}[theorem]{Example}
\newtheorem{assumption}[theorem]{Hypothesis}

\title{Integral representations of lower semicontinuous envelopes and Lavrentiev Phenomenon for non continuous Lagrangians}
\author{Tommaso Bertin}

\begin{document}

\maketitle
\begin{abstract}

We consider the functional
\[
F_\infty(u)= \int_{\Omega}f(x,u(x),\nabla u(x)) dx\quad u\in \varphi+ W_0^{1,\infty}(\Omega,\R)
\]
where $\Omega$ is an open bounded Lipschitz subset of $\R^N$ and $\varphi\in W^{1,\infty}(\Omega)$. 
We do not assume neither convexity or continuity of the Lagrangian w.r.t. the last variable. We prove that, under suitable assumptions,  the lower semicontinuous envelope of $F_\infty$ both in $\varphi+W^{1,\infty}(\Omega)$ and in the larger space $\varphi+W^{1,p}(\Omega)$ can be represented by means of the bipolar $f^{**}$ of $f$.  
In particular we can also exclude Lavrentiev Phenomenon between $W^{1,\infty}(\Omega)$ and $W^{1,1}(\Omega)$ for autonomous Lagrangians.
\end{abstract}

\section{Introduction}
We consider the functional \begin{equation}
 \label{functionalnew}
 F_q(u)= \int_{\Omega}f(x,u(x),\nabla u(x)) dx\qquad u\in \varphi+ W_0^{1,q}(\Omega)
\end{equation}
where $1\leq q \leq\infty$, $\Omega$ is an open bounded subset of $\R^N$, $\varphi\in W^{1,q}(\Omega)$ and $f:\Omega\times\R\times\R^N$ is a suitable Borel function that is not necessarily assumed to be convex in the last variable. Due to the lack of convexity of the Lagrangian, the functional is not sequentially lower semicontinuous with respect to the weak topology in the case $1\le q<\infty$, resp weak$^*$ topology in the case $q=\infty$. We address the problem of representing the sequential weak lower semicontinuous envelope of $F_q$.
\\
To be more precise, for every $1\le p\le q$, we define the functional 
\begin{equation}
     \overline{F_{p}}(u)=\begin{cases} F_q(u)&\text{ if } u\in \varphi+ W^{1,q}_0(\Omega)\\
 +\infty &\text{ if } u\in \varphi+ W^{1,p}_0(\Omega) \setminus W^{1,q}_0(\Omega) 
 \end{cases} 
\end{equation}
and we are interested in determining sufficient conditions for the following identity to hold
\begin{equation}\label{relaxed}
\text{sc}^-(\overline{F_{p}})(u)=\int_{\Omega}f^{**}(x,u(x),\nabla u(x)) dx\qquad \forall u\in \varphi+ W_0^{1,p}(\Omega)
\end{equation}
where $\text{sc}^-(\overline{F_{p}})(u)$ denotes the greatest sequentially weak (weak$^*$ in the case of $p=\infty$) lower semicontinuos functional, with respect to $W^{1,p}(\Omega)$, that is less or equal to $\overline{F_{p}}$ and $f^{**}$ is the convexified function of $f$ w.r.t. the last variable. 
In the literature this problem has been studied from two different points of view. From one hand if $f$ is not convex with respect to the last variable then $F_q$ is not weakly lower semicontinuous and it is interesting to consider $sc^-(F_q)$. On the other hand, even in the case where $f$ is convex w.r.t. the last variable, it is important to represent $sc^-(\overline{F_p})$ for $p<q$. In this paper we put together the two approach.
Despite the fact that the identity in \eqref{relaxed} seems very natural, we know that it does not hold true, for the case $p<q$, even for 'very regular' functionals. In fact every functional exhibiting the so called {\it Lavrentiev phenomenon} (introduced for the first time by Lavrentiev in \cite{Lavrentiev}) cannot satisfy the identity \eqref{relaxed}. In the classic example by Manià (\cite{Manià1934}) it has been shown that 
\begin{equation}
\min_{\text{id}+W_0^{1,1}([0,1])}\int_0^1 (x-u^3(x))^2 |u'(x)|^6\,dx
< \inf_{\text{id}+W_0^{1,\infty}([0,1])}\int_0^1 (x-u^3(x))^2 |u'(x)|^6\,dx    
\end{equation}
and we notice that, in this case, the Lagrangian is not only smooth in the three variables, but it is also convex w.r.t. the derivative variable. Further examples in which the minimum, or the infimum, in 
$u\in \varphi+ W^{1,p}_0(\Omega)$ is strictly less than the minimum, or the infimum, in $u\in \varphi+ W^{1,q}_0(\Omega)$, with $p<q$, can be found for one-dimensional case in \cite{BallMizel1985}, \cite{Mizel1988} and for multidimensional case in \cite{Zhykov1995}. In the present paper we focus our attention on the scalar multidimensional case but we have to mention, for the sake of completeness, that there are examples also in the vectorial case (see \cite{ButtazzoBelloni1995} for a wide list of examples). The problem of detecting conditions that prevent Lavrentiev phenomenon is important in particular for numerical approximations and engineering applications. The problem is well studied in dimension 1 with very weak hypotheses about the Lagrangian (\cite{Alberti1994}, \cite{Mariconda2023}). In higher dimension usually Lagrangians are assumed convex  (\cite{BMT2014}, \cite{Bousquet2023}, \cite{Bousquet2024}, \cite{Bousquet2025}).

The approach to the Lavrentiev phenomenon as a problem of representation of the relaxed functional has been considered for the first time, as far as we know, in \cite{ButtazzoMizel1992} where the authors introduced the notion of Lavrentiev gap at a fixed $u\in\varphi+W^{1,p}_0(\Omega)$. Precisely they say that the {\it Lavrentiev gap} occurs at $u$, for convex Lagrangians, when the difference between $\text{sc}^-(\overline{F_{p}})(u)$ and $\int_{\Omega}f(x,u(x),\nabla u(x)) dx$ is strictly positive. 

Starting from \cite{ET} and \cite{MS} the problem of integral representation for $sc^-(\overline{F_p})$ was investigated by many authors (\cite{ButtazzoDalMaso1980}, \cite{ButtazzoDalMasoDeGiorgi1983}, \cite{ButtazoDalMaso1985}, \cite{ButtazzoLeaci1985}, \cite{Mizel1988}, \cite{DalMaso1988}
\cite{Buttazzo1989}). Many  of these papers are devoted to avoid the assumption of continuity of the Lagrangian with respect to $u$ and study the property of the functional $sc^-(\overline{F_p})$ over different subsets of $\Omega$.
The main goal of our paper is to prove the validity of \eqref{relaxed} in the scalar multidimensional case without hypothesis of continuity with respect to $\xi$ for non autonomous Lagrangians in the case $p=q=\infty$ (cfr Teorem \ref{pro gen}) and for autonomous Lagrangians in the case $p=1, q=\infty$ (cfr Theorem \ref{Theo fin}).

In section \ref{sezione solo x e gradiente} we start considering Lagrangians depending by $x$ and $\nabla u$. We take inspiration by the constructive method by Ekeland and Temam in \cite{ET} to find for every $u\in\varphi+ W^{1,\infty}_0(\Omega)$ a function $v\in\varphi+ W^{1,\infty}_0(\Omega)$ sufficiently near to $u$ with $F_{\infty}(v)$ sufficiently near to $\int_{\Omega}f^{**}(x,\nabla u(x)) dx$.
We modified the proof presented in \cite[Proposition 3.2, page 330]{ET} to deal with the absence of continuity of $f$ w.r.t. $\nabla u$. We notice, in particular, the fact that our construction uses the Vitali covering Theorem.
The main advantage of this argument is that it allows us to construct an explicit sequence $u_n$ converging to $u$ in $L^{\infty}(\Omega)$ and such that the value $F_{\infty}(u_n)$ converges to $\int_{\Omega}f^{**}(x,\nabla u(x)) dx$. We deduce also a first relaxation result and the validity of (\ref{relaxed}) in the case $p=q=\infty$.

In section \ref{sezionegenerale} we 
still consider the case $p=q=\infty$ and we extend the results of section \ref{sezione solo x e gradiente} to the case of a general Lagrangian $f(x,u,\xi)$ satisfying a suitable set of assumption that we will denote by Hypothesis \ref{gen ipo}.
In particular we use a truncation method for $f$ with respect to the variable $\xi$ considering the auxiliary function $f_K(x,u,\xi)$ equal to $f(x,u,\xi)$ if $\|\xi\|\leq K$ and $+\infty$ otherwise. It is interesting to note that for $f_K$ it is possible to find a sequence $u_{n,K}\overset{*}{\rightharpoonup} u$ such that $\int_{\Omega}f_K(x,u_{n,K}(x),\nabla u_{n,K}(x)) dx$ converges to $\int_{\Omega}f^{**}_K(x,u(x),\nabla u(x)) dx$. Later on, for $f(x,u,\xi)$ we pass to the limit $K\to \infty$ and it can happen, in general, that it does not exist a sequence such that $u_n\overset{*}{\rightharpoonup} u$ and $F_{\infty}(u_n)$ converges to $\int_{\Omega}f^{**}(x,u(x),\nabla u(x)) dx$.

 In section \ref{sezioneLav} we observe that under our hypotheses, the non occurrence of Lavrentiev Phenomenon for the relaxed functional $\int_{\Omega}f^{**}(x,u(x),\nabla u(x))dx$ implies the non occurrence of Lavrentiev Phenomenon for the original functional.
 
In section \ref{sezione auto} we firstly apply the general results of sections \ref{sezionegenerale} and \ref{sezioneLav} to autonomous Lagrangians. This is a case of great interest in many applications of the Calculus of Variations and has the peculiarity that the assumptions in these case appear more natural. We focus on this family of Lagrangians to apply a recent result by Bousquet (\cite{Bousquet2023}) which prove, in the case of autonomous, continuous and convex Lagrangian, for every $u\in \varphi+ W^{1,1}_0(\Omega)$ the existence of a sequence $u_n\in \varphi+W^{1,\infty}_0(\Omega)$ such that $u_n\to u$ in $W^{1,1}$ and $F_{\infty}(u_n)$ converges to $F_1(u)$. This theorem allows us, under suitable hypothesis, to exclude the Lavrentiev phenomenon for autonomous non convex Lagrangians and to prove the validity of the \eqref{relaxed} on the case where $q=\infty$ and $p=1$.
The validity of \eqref{relaxed} with this special choice of $p$ and $q$ implies that that the value of the functional
\[
\int_{\Omega}f^{**}(x,u(x),\nabla u(x)) dx
\]
can be approximated evaluating 
\[
\int_{\Omega}f(x,u(x),\nabla u(x)) dx
\]
on a suitable sequence of Lipschitz functions. This fact has a great impact on numerical estimates of the functional since the Finite Element Method, for example, relies on the use of Lipschitz functions.

Future developments will be to find some general geometrical condition on Lagrangian $
f$ to apply the result of (\cite{Bousquet2023}) to $f^{**}$ and try to extend the relaxation results to the non autonomous cases.
For the sake of completeness we cite an interesting paper (\cite{Bousquet2025}), which has to appear, about an approximation result similar to (\cite{Bousquet2023}) in the case of non autonomous Lagrangians continuous and convex.

\section{An approximation result for Lagrangians depending only on $x$ and $\nabla u$}
\label{sezione solo x e gradiente}
In this section we consider a Lagragian $\Tilde f(x,\xi)$ which satisfies following Hypotheses.

\begin{assumption}
\label{Ipo caso x}
Let $\Omega$ be an open bounded Lipschitz subset of $\R^N$, $B_K(0)\subset\R^n$ the open ball with center $0$ and radius $K>0$ and $\Tilde{f}:\Omega\times B_K(0)\to\R$ be a function such that
\begin{itemize}
\item[a)] $\Tilde{f}(x,\xi):\Omega\times B_K(0)\to\R$ coincides a.e. with a borelian function;
\item[b)] there exists $ a\in L^1(\Omega)$ such that 
    $0\leq \Tilde{f}(x,\xi)\leq a(x)$  for every $ \xi\in B_K(0)$\,;
    \item[c)] for all $ \delta>0$ then there exists $ T\subset \Omega$ compact such that $|\Omega\setminus T|<\delta$ and $\Tilde{f}(x,\xi)_{|T\times B_K(0)}$ is
continuous with respect to $x$ uniformly as $\xi$ varies in $B_K(0)$\,.
\end{itemize}
\end{assumption}
\begin{remark}
   We do not assume any continuity for $\Tilde f$ with respect to $\xi$. This is the main novelty with respect to \cite{ET}, where the Lagrangian is assumed to be a Caratheodory function, and to \cite{MS}, where the Lagrangian is upper semicontinuous with respect to $\xi$.
\end{remark}

The next two lemmas state some properties directly implied by Hypothesis \ref{Ipo caso x}.

\begin{lemma}
\label{unif cont case x}
    Let $\Tilde f:\Omega\times B_K(0)\rightarrow\R$ satisfy Hypothesis \ref{Ipo caso x} and let $T\subset\Omega$ be a compact set such that Hypothesis 1 c) is satisfied. Then $\Tilde f(x,\xi)$ is uniformly continuous as $x$ varies in $T$ uniformly as $\xi$ varies in $B_K(0)$.
\end{lemma}
\begin{proof}
    Given $\varepsilon>0$ for every $x\in T$ there exists a $\eta_x$ such that
\begin{equation}
    |\Tilde f(\Tilde{x},\xi)-f(x,\xi)|<\varepsilon
\end{equation}
for every $\Tilde{x}\in ]x-\eta_x, x+\eta_x[\cap T$ and for every $\xi\in B_K(0)$\,.

Since $\{]x-\eta_x, x+\eta_x[\cap T\}_x$ is an open covering of $T$, which is compact, we can extract a finite subcovering and in particular there exists a $\eta>0$ such that
\begin{equation}
    |\Tilde f(x_1,\xi)-\Tilde f(x_2,\xi)|<\varepsilon
\end{equation}
for every $x_1,x_2\in T$ such that $|x_1-x_2|<\eta$ and for every $\xi\in B_K(0)$.

\end{proof}

Given a function $f(x,u,\xi)$ we indicate with $f^{**}(x,u,\xi)$ the bipolar of $f$ with respect to $\xi$. Actually $f^{**}(x,u,\xi)$ is the biggest function lower semicontinuous and convex with respect to $\xi$ lower or equal than $f(x,u,\xi)$ (cfr \cite[Proposition 4.1, pag 18]{ET}).

In the next lemma we prove that Hypothesis \ref{Ipo caso x} implies  that $\tilde f^{**}$ is continuous on $T\times B_K(0)$ for every $T$ such that Hypothesis \ref{Ipo caso x} $c)$ holds and that $\tilde f^{**}$ is a Borel function on $\Omega\times B_K(0)$.

\begin{lemma}
    \label{fstarcont case x} 
If $\Tilde f:\Omega\times B_K(0)\to\R$ satisfies Hypothesis \ref{Ipo caso x} then $\Tilde f^{**}$ is continuous on $T\times B_K(0)$ for every $T$ as in Hypothesis \ref{Ipo caso x} c).
Furthermore $\Tilde f^{**}$ is a borelian function on $\Omega\times B_K(0)$.
\end{lemma}
\begin{proof}
 First of all we prove that $\Tilde f^{**}(x,\xi)$ is continuous w.r.t. $x$ in $T$ uniformly as $\xi$ varies in $B_K(0)$. We fix $x_0\in T$ and for every $\varepsilon>0$ there exists $\eta>0$ such that if $x,x_0\in T$ and $|x-x_0|<\eta$ then
\begin{equation}
    |\Tilde f(x,\xi)-\Tilde f(x_0,\xi)|<\varepsilon \quad \forall\xi\in B_K(0)
\end{equation}
and so we have
\begin{equation}
\label{cont uni case x}
    \Tilde f^{**}(x,\xi)-\varepsilon\leq \Tilde f(x,\xi)-\varepsilon\leq \Tilde f(x_0,\xi)\,.
\end{equation}
Now $\Tilde f^{**}(x,\xi)-\varepsilon$ is convex in $\xi$ and and so 
\begin{equation}
    \Tilde f^{**}(x,\xi)-\varepsilon\leq \tilde f^{**}(x_0,\xi)\,.
\end{equation}
By Lemma \ref{unif cont case x}  $\Tilde f(\cdot,\xi)$ is uniformly continuous in $T$ uniformly as $\xi$ varies in $B_K(0)$, thus we can change the roles of
 $x$ and $x_0$ in the previous inequalities and we obtain 
\begin{equation}
    |\Tilde f^{**}(x,\xi)-\Tilde f^{**}(x_0,\xi)|<\varepsilon \quad \forall\xi\in B_K(0)\,.
\end{equation} 
 Since $\Tilde f^{**}(x,\cdot)$ is continuous in $B_K(0)$ for every $x$, for every sequence $(x_n,\xi_n)\in T\times B_K(0)$ converging to $(x_0,\xi_0)\in T\times B_K(0)$ we have that
\begin{align}
\label{cont2var case x}
\lim_{n}\,|\Tilde f^{**}(x_n,\xi_n)&-\Tilde f^{**}(x_0,\xi_0)|\\
     \leq &\lim_{n}|\Tilde f^{**}(x_n,\xi_n)-\Tilde f^{**}(x_0,\xi_n)|+\lim_{n}|\Tilde f^{**}(x_0,\xi_n)-\Tilde f^{**}(x_0,\xi_0)|=0   
\end{align}
i.e. $\Tilde f^{**}$ is continuous on $T\times B_K(0)$.

Now, recalling  that
\begin{equation}
    f^{**}(x,\xi)=\lim_n(f_{|T_n\times B_K(0)})^{**}(x,\xi)
\end{equation}
with $|\Omega\setminus T_n|\to 0$, the continuity of $(f_{|T_n\times B_K(0)})^{**}$ implies that $f^{**}$ is borelian on $\Omega\times\R^N$. 

\end{proof}

The next Theorem is inspired by the analogous one in  \cite[Chapter X, Proposition 3.2]{ET}. The main novelty here is that we do not assume continuity of the Lagrangian with respect to the variable $\xi$.  Af far as we know, this is the first case in literature in which the identity \eqref{relaxed} about integral representation of $sc^-({F_{\infty})}$ is proved without assuming continuity with respect to $\xi$.

\begin{theorem}
\label{lemma}
Let $\Omega$ be an open bounded Lipschitz subset of $\R^N$, $K\in\N$ and $\Tilde{f}:\Omega\times B_K(0)\to\R$ satisfiy Hypothesis \ref{Ipo caso x}.\\
Then for every $u\in W^{1,\infty}(\Omega)$ such that 
\begin{equation}
    \|\nabla u\|_{\infty}<K
\end{equation}
there exists a sequence $v_{n}\in u+W^{1,\infty}_0(\Omega)$ such that
\begin{equation}
    \|\nabla v_{n}\|_{\infty}<K\,,
\quad \quad
 \lim_{n}\|v_{n}-u\|_{\infty}\to 0
\end{equation}
and
\begin{equation}
   \lim_{n}\Big{|}\int_{\Omega}\Tilde{f}(x,\nabla v_{n}(x)) dx-\int_{\Omega} \Tilde{f}^{**}(x,\nabla u(x))dx  \Big{|}\to 0\,.
\end{equation}    
Furthermore
\begin{equation}
    sc^-(F_{\infty})(u)=\int_{\Omega}\Tilde{f}^{**}(x,\nabla u(x))dx
\end{equation}
where with $sc^-(F_{\infty})$ we denote the lower semicontinuous envelope of 
\begin{equation}
    F_{\infty}(u)=\int_{\Omega}\Tilde{f}(x,\nabla u(x))dx
\end{equation}
with respect to the weak$^*$ topology of $W^{1,\infty}(\Omega)$. 
\end{theorem} 

\begin{proof}

{\it Step 1.} First of all we observe by Lemma \ref{fstarcont case x} that $\Tilde f^{**}(x,\xi)$ coincides a.e. with a borelian function on $\Omega\times B_K(0)$.

{\it Step 2.} For the first part is sufficient to show that for every $\varepsilon>0$ there exists $v_{\varepsilon}\in u+W^{1,\infty}_0(\Omega)$ with $\|\nabla v_{\varepsilon}\|<K$ such that
\begin{equation}
   \|u-v_{\varepsilon}\|_{\infty}<\varepsilon, \qquad \Big |\int_{\Omega}\Tilde{f}(x,\nabla v_{\varepsilon}(x)) dx-\int_{\Omega} \Tilde{f}^{**}(x,\nabla u(x))dx\Big |<\varepsilon\,.
\end{equation}
Following the same approach as in \cite{ET} we start considering the case where $u$ is affine, i.e. $\nabla u=\overline{\xi}$ and we construct a Lipschitz function sufficiently close to $u$, preserving the boundary datum and  such that its gradient is a.e. in a suitable set. 

For every $x\in\Omega$ we can write $\Tilde{f}^{**}(x,\overline{\xi})$ as \cite[Remark 3.27]{Dac}
\begin{equation}
    \Tilde{f}^{**}(x,\overline{\xi})=\inf\left\{\sum_{i=1}^{n+1}\alpha^x_i \Tilde f(x,\xi^x_i)\Big|\,
    \alpha^x_i\geq 0,  \xi^x_i\in B_K(0),\sum_{i=1}^{n+1} \alpha^x_i=1,\,  \sum_{i=1}^{n+1}\alpha^x_i\xi_i^x=\overline{\xi}\right\}
\end{equation}
so that, for $\varepsilon>0$, we can choose $\alpha^x_i\ge 0$, $ \xi^x_i\in B_K(0)$ such that $\sum_{i=1}^{n+1} \alpha^x_i=1$,  $\sum_{i=1}^{n+1}\alpha_i^x\xi_i^x=\overline{\xi}$ and 
\begin{equation}
\sum_{i=1}^{n+1}\alpha^x_i \Tilde f(x,\xi^x_i)-\Tilde{f}^{**}(x,\overline{\xi})< \frac{\varepsilon}{9|\Omega|}\,.
\end{equation}
We fix $0<\delta<\varepsilon$ such that for every $\omega\subset \Omega$ with $|\omega|<\delta$ then
\begin{equation}
    \int_{\omega}a(x)dx<\frac{\varepsilon}{24}
\end{equation}
where $a\in L^1(\Omega)$ satisfies Hypothesis \ref{Ipo caso x} $b)$.

Hypothesis \ref{Ipo caso x} $c)$ implies there exists $T\subset\Omega$ compact such that  $\Tilde{f}^{**}_{|T\times B_K(0)}(\cdot,\overline{\xi})$ is uniformly continuous (since $T$ is compact),  $\Tilde{f}_{|T\times B_K(0)}(x,\xi)$ is uniformly
continuous with respect to $x$ uniformly as $\xi$ varies in $B_K(0)$ and $|\Omega\setminus T|<\delta$ so that
\begin{equation}
\label{stima a T}
    \int_{\Omega\setminus T} a(x)dx<\frac{\varepsilon}{24}\,.
\end{equation}
Moreover for every $x\in T$ there exists a neighbourhood of $x$, $U_x\subset \Omega$, such that 
\begin{equation}
\label{intorni}
    \forall y\in U_x\cap T\quad\quad |\Tilde{f}(y,\xi)-\Tilde{f}(x,\xi)|\leq \frac{\varepsilon}{9 |\Omega|}\quad \forall \xi\in B_K(0)
\end{equation}
and 
\begin{equation}
    \forall y\in U_x\cap T\quad\quad |\Tilde{f}^{**}(y,\overline{\xi})-\Tilde{f}^{**}(x,\overline{\xi})|\leq \frac{\varepsilon}{9 |\Omega|}\,.
\end{equation}
In particular we have 
\begin{equation}
\label{stima com conv}
      \forall y\in U_x\cap T\quad\quad 
      \bigg|\sum_{i=1}^{n+1}\alpha^x_i \Tilde{f}(y,\xi^x_i)-\Tilde{f}^{**}(y,\overline{\xi})\bigg|
      \leq \frac{\varepsilon}{3 |\Omega|}\,.
\end{equation}
Now, for every $x\in T$, there exists a regular family of closed balls of center $x$ and radius $r$ denoted as $\overline {B}_r(x)\subset U_x$, $0<r<r_x$, that covers $T$ in the sense of Vitali. Then we can apply Vitali Covering theorem (cfr \cite[Chapter IV, 3, page 109]{S}) to find a countable family of 
\begin{equation}
\label{def omega}
    \omega^j=\overline{B}_{r_j}(x_j)
\end{equation}
such that
\begin{equation}
\big|T\setminus \bigcup_j  \omega_j\big |=0\,,\quad \quad    \omega^{j_1}\cap \omega^{j_2}=\emptyset\quad j_1\neq j_2\,.
\end{equation}
We remark that in general
\begin{equation}
    T\subsetneq \bigcup_j \omega^j\,.
\end{equation}
By \cite[Chapter X, Theorem 1.2, pag. 300]{ET} we deduce that for every $\omega^j$ and
\begin{equation}
    0<\delta_j<\frac{\delta}{2^{j}}
\end{equation}
 we can find $n+1$ subsets of $\omega^j$, $\omega_i^j$ and a locally Lipschitz function $v_j$ such that 
\begin{gather}\label{prop:vj}
    \big||\omega_i^j|-\alpha^{x_j}_i |\omega^j|\big|\leq\alpha^{x_j}_i \delta_j\quad \text{for}\quad 1\leq i\leq n+1\,,\\
    \nabla v_j=\xi^{x_j}_i\quad \text{on} \quad \omega_i^j\,,\\
     \|\nabla v_j\|_{\infty}< K\quad \text{on}  \quad \omega^j\,,\\
      \|v_j-u\|_{\infty}\leq\delta\quad \text{on} \quad \omega^j\,,\\
      v_j=u\quad \text{on} \quad \partial\omega^j\,.
\end{gather}
In particular the first property implies 
\begin{equation}
\label{differenza omega}
   \big||\omega^j|-|\bigcup_i\omega^j_i|\big|\leq  \big||\omega^j|-\sum_i|\omega^j_i|\big|\leq \frac{\delta}{2^{j}}\,.
\end{equation}
We define the function
\begin{equation}
    v_{\varepsilon}(x):=
    \begin{cases}
         v_j(x) \quad \text{if} \quad x\in \omega^j\,, \\
         u(x) \quad \text{if} \quad x\in \Omega\setminus\bigcup\limits_{j=1}^{\infty}\omega^j
    \end{cases}
\end{equation}
and it easily turns out that
$v_{\varepsilon}\in u+W_0^{1,\infty}(\Omega)$,  $\|\nabla v_{\varepsilon}\|_{\infty}<K$ and, since
 $\delta< \varepsilon$,
\begin{equation}
    \|v_{\varepsilon}-u\|_{\infty}<\varepsilon\,.
\end{equation}

{\it Step 3.} Our aim now is to evaluate
\begin{equation}
   \Big{|}\int_{\Omega}\Tilde{f}(x,\nabla v_{\varepsilon}(x)) dx-\int_{\Omega} \Tilde{f}^{**}(x,\nabla u(x))dx  \Big{|}.
\end{equation}
We start observing that
\begin{align}
\label{stima A}
     \Big{|}\int_{\Omega}&\Tilde{f}(x,\nabla v_{\varepsilon}(x)) dx-\int_{\Omega} \Tilde{f}^{**}(x,\nabla u(x))dx  \Big{|}
      \\\leq&
     \ \int_{\Omega\setminus \bigcup\limits_{j=1}^{\infty}\omega^j}|\Tilde{f}(x,\nabla v_{\varepsilon}(x))-\Tilde{f}^{**}(x,\nabla u(x))|dx
    \\& +\sum_{j=1}^\infty
     \Big{|}\int_{\omega^j}\Tilde{f}(x,\nabla v_{\varepsilon}(x)) dx-\int_{\omega^j} \Tilde{f}^{**}(x,\nabla u(x))dx  \Big{|}\,   
\end{align}
 Since $|T\setminus\cup_{j=1}^\infty \omega_j|=0$, we can estimate the  first term in the right hand side of \eqref{stima A} using (\ref{stima a T}) and recalling that, by assumption $b)$, $f^{**}$ is non negative. We then obtain

\begin{equation}
    \int_{\Omega\setminus \bigcup\limits_{j=1}^{\infty}\omega^j}|\Tilde{f}(x,\nabla v_{\varepsilon}(x))-\Tilde{f}^{**}(x,\nabla u(x))|dx\leq
    \int_{\Omega\setminus T}a(x)dx<\frac{\varepsilon}{24}\,.
\end{equation}
We consider now each term in the last sum of  \eqref{stima A}.
We have that
\begin{align}
\label{omegaj}
   \Big{|}\int_{\omega^j}&\Tilde{f}(x,\nabla v_{\varepsilon}(x)) dx-\int_{\omega^j} \Tilde{f}^{**}(x,\nabla u(x))dx\Big{|}
   \\
\leq&\Big{|}\int_{\omega^j}\Tilde{f}(x,\nabla v_{\varepsilon}(x)) dx-\sum_{i=1}^{n+1}\alpha^{x_j}_i 
    \int_{\omega^j}\Tilde{f}(x,\xi^{x_j}_i)dx\Big{|}\\
&+\Big{|}\sum_{i=1}^{n+1}\alpha^{x_j}_i \int_{\omega^j}\Tilde{f}(x,\xi^{x_j}_i)dx-
   \int_{\omega^j} \Tilde{f}^{**}(x,\nabla u(x))dx\Big{|}\\ 
\end{align}
and, recalling the definition of the functions $v_\varepsilon$ and $v_j$, we can estimate the right hand side in \eqref{omegaj} as
\begin{align}\label{omegaj2}
 &\Big{|}\int_{\omega^j}\Tilde{f}(x,\nabla v_{j}(x)) dx-\sum_{i=1}^{n+1}\alpha^{x_j}_i 
    \int_{\omega^j}\Tilde{f}(x,\xi^{x_j}_i)dx\Big{|}\\
&+\Big{|}\sum_{i=1}^{n+1}\alpha^{x_j}_i \int_{\omega^j}\Tilde{f}(x,\xi^{x_j}_i)dx-
   \int_{\omega^j} \Tilde{f}^{**}(x,\nabla u(x))dx\Big{|}\\
\le&\Big{|}\sum_{i=1}^{n+1} 
   \left[\int_{\omega^j_i}\Tilde{f}(x,\xi^{x_j}_i)dx-\alpha^{x_j}_i 
    \int_{\omega^j}\Tilde{f}(x,\xi^{x_j}_i)dx\right]\Big{|}\\
&+\int_{\omega^j\setminus\bigcup_{i}\omega_i^j}\Tilde{f}(x,\nabla v_j(x))dx\\
&+\Big{|}\sum_{i=1}^{n+1}\alpha^{x_j}_i \int_{\omega^j}\Tilde{f}(x,\xi^{x_j}_i)dx-
   \int_{\omega^j} \Tilde{f}^{**}(x,\nabla u(x))dx\Big{|}\,.
\end{align}
We consider each term in the first sum of the right hand side of (\ref{omegaj2}) and we obtain
\begin{align}
\label{stima affine}
  & \bigg| 
   \int_{\omega^j_i}\Tilde{f}(x,\xi^{x_j}_i)dx-
\alpha^{x_j}_i\int_{\omega^j}\Tilde{f}(x,\xi^{x_j}_i)dx\bigg|
\\
 & \leq \bigg|\int_{\omega^j_i}\Tilde{f}(x,\xi^{x_j}_i)-\Tilde{f}(x_j,\xi^{x_j}_i)dx\bigg|
\\+
&\bigg| \int_{\omega^j_i}\Tilde{f}(x_j,\xi^{x_j}_i)dx-
\alpha^{x_j}_i  \int_{\omega^j}\Tilde{f}(x_j,\xi^{x_j}_i)dx\bigg|\\
&+\alpha^{x_j}_i\bigg|\int_{\omega^j}\Tilde{f}(x_j,\xi^{x_j}_i)-\Tilde{f}(x,\xi^{x_j}_i)dx\bigg|\,.
\end{align}
Since  $\Tilde{f}(x_j,\xi^{x_j}_i)$ is constant, by the first property in (\ref{prop:vj}), we can estimate the second term in the right hand side as follows
\begin{equation}
\label{const}
  \bigg| \int_{\omega^j_i}\Tilde{f}(x_j,\xi^{x_j}_i)dx-
\alpha^{x_j}_i  \int_{\omega^j}\Tilde{f}(x_j,\xi^{x_j}_i)dx\bigg|
\leq  \alpha^{x_j}_i\delta_j|\Tilde{f}(x_j,\xi^{x_j}_i)|\,.
\end{equation}
In order to evaluate the other two terms in \eqref{stima affine} we notice that, by the definition of $\omega^j$ in (\ref{def omega}) and the property (\ref{intorni}) 
\begin{equation}
    \forall x\in \omega_j\cap T\quad |\Tilde f(x,\xi^{x_j}_i)-\Tilde f(x_j,\xi^{x_j}_i)|\leq \frac{\varepsilon}{9 |\Omega|}
\end{equation}
and so 
\begin{equation}
   \bigg|\int_{\omega^j_i\cap T}\Tilde{f}(x,\xi^{x_j}_i)dx-
\int_{\omega^j_i\cap T}\Tilde{f}(x_j,\xi^{x_j}_i)dx\bigg|\leq \frac{\varepsilon |\omega^j_i|}{9|\Omega|}\,,
\end{equation}
\begin{equation}
    \bigg| \alpha^{x_j}_i 
   \int_{\omega^j\cap T}\Tilde{f}(x_j,\xi^{x_j}_i)dx
    -
\alpha^{x_j}_i 
   \int_{\omega^j\cap T}\Tilde{f}(x,\xi^{x_j}_i)dx\bigg|\leq \frac{\varepsilon \alpha_i^{x_j}|\omega^j|}{9|\Omega|}\,.
\end{equation}
It follows that 
\begin{multline}
\label{primo termine destro}
    \bigg| 
   \int_{\omega^j_i}\Tilde{f}(x,\xi^{x_j}_i)dx-
\alpha^{x_j}_i\int_{\omega^j}\Tilde{f}(x,\xi^{x_j}_i)dx\bigg|
\leq\\
 \frac{\varepsilon |\omega^j_i|}{9|\Omega|}+\int_{\omega_i^j\setminus T}a(x)+\alpha^{x_j}_i \delta_j|\Tilde{f}(x_j,\xi^{x_j}_i)|+
\frac{\varepsilon \alpha_i^{x_j}|\omega^j|}{9|\Omega|}+
\alpha_i^{x_j}\int_{\omega^j\setminus T}a(x)\,.
\end{multline}
Coming back to the sum on $i$ in \eqref{omegaj2} and recalling the first property of (\ref{prop:vj}) we have
\begin{multline}
    \bigg| \sum_{i=1}^{n+1}\Bigg[
   \int_{\omega^j_i}\Tilde{f}(x,\xi^{x_j}_i)dx-\alpha^{x_j}_i
\int_{\omega^j}\Tilde{f}(x,\xi^{x_j}_i)dx\Bigg]\bigg|
\leq
\\
\delta_j\max_i\bigg\{|\Tilde{f}(x_j,\xi^{x_j}_i)|\bigg\}+
\frac{\varepsilon(2 |\omega^j|)}{9|\Omega|}+\frac{\varepsilon\delta_j}{9|\Omega|}+\int_{\omega^j\setminus T}a(x)+\int_{\bigcup_i\omega^j_i\setminus T}a(x)\,.
\label{stima unione}
\end{multline}
Since $|\Tilde f(x,\xi)|\leq a(x)$ and $\Tilde{f}_{|T\times B_K(0)}(x,\xi)$ is uniformly continuous in $x$ uniformly as $\xi$ varies in $ B_K(0)$, then $\Tilde{f}_{|T\times B_K(0)}(x,\xi)$ is bounded. In fact fixed $\overline x$ then $\Tilde f(\overline x,\xi)$ is bounded for every $\xi\in B_K(0)$ and 
it exists a modulus of continuity such that
\begin{equation}
    |\Tilde{f}(x,\xi)-\Tilde{f}(\overline x,\xi)|\leq \gamma (|x-\overline x|)
\end{equation}
for every $x\in T$ and $\xi\in B_K(0)$. Thus there exists $C>0$ which does not depend by $x_j$ or $\xi^{x_j}_i$ such that
\begin{equation}
    \max_i\bigg\{|\Tilde{f}(x_j,\xi^{x_j}_i)|\bigg\}\leq C\,.
\end{equation}
Now we note that $C$ depends only by the choose of the compact $T$. So we can take $\delta_j$ sufficiently small such that 
\begin{equation}
     \frac{2\varepsilon |\omega^j|}{9|\Omega|}+\delta_j C+\frac{\varepsilon\delta_j}{9|\Omega|}\leq \frac{\varepsilon|\omega^j|}{3|\Omega|}\,.
\end{equation}
Now \eqref{stima unione}
becomes
\begin{align}
   \bigg| \sum_{i=1}^{n+1}\Bigg[
   \int_{\omega^j_i}\Tilde{f}(x,\xi^{x_j}_i)dx-\alpha^{x_j}_i
\int_{\omega^j}\Tilde{f}(x,\xi^{x_j}_i)dx\Bigg]\bigg|
\leq\\
\frac{\varepsilon|\omega^j|}{3|\Omega|}
+\int_{\omega^j\setminus T}a(x)+\int_{\bigcup_i\omega^j_i\setminus T}a(x)
\label{stima primo pezzo}
\end{align}
and we have completed the estimate of the first term in the right hand of \eqref{omegaj2}.

\noindent Now we turn to evaluate the second term.
By assumption $b)$ 
\begin{equation}
       \int_{\omega^j\setminus\bigcup_i \omega_i^j}\Tilde{f}(x,\nabla v_j(x))dx\leq 
          \int_{\omega^j\setminus\bigcup_i \omega_i^j}a(x)dx
\end{equation}
where we recall by (\ref{differenza omega}) that
\begin{equation}
   \big||\omega^j|-|\bigcup_i\omega^j_i|\big|\leq \frac{\delta}{2^{j}}\,.
\end{equation}

\noindent We split the third term in the right hand side of \eqref{omegaj2} 
\begin{multline}
    \label{secondo pezzo omega j}
    \Big|\sum_{i=1}^{n+1}\alpha^{x_j}_i 
   \int_{\omega^j}\Tilde{f}(x,\xi^{x_j}_i)dx-
   \int_{\omega^j} \Tilde{f}^{**}(x,\nabla u(x))dx\Big{|}\\
   \leq
   \Big|\sum_{i=1}^{n+1}\alpha^{x_j}_i 
   \int_{\omega^j\cap T}\Tilde{f}(x,\xi^{x_j}_i)dx-
   \int_{\omega^j\cap T} \Tilde{f}^{**}(x,\nabla u(x))dx\Big{|}
   +\int_{\omega^j\setminus T}a(x)dx\,.
\end{multline}
By (\ref{stima com conv}) we have 
\begin{equation}
    \Big|\sum_{i=1}^{n+1}\alpha^{x_j}_i 
   \int_{\omega^j\cap T}\Tilde{f}(x,\xi^{x_j}_i)dx-
   \int_{\omega^j\cap T} \Tilde{f}^{**}(x,\nabla u(x))dx\Big{|}\leq 
   \frac{\varepsilon|\omega^j|}{3|\Omega|}\,.
\end{equation}
Considering together these tree evaluations we can rewrite \eqref{omegaj} as
\begin{align}
   \Big{|}\int_{\omega^j}&\Tilde{f}(x,\nabla v_{\varepsilon}(x)) dx-\int_{\omega^j} \Tilde{f}^{**}(x,\nabla u(x))dx\Big{|}
   \\
\leq&\Big{|}\int_{\omega^j}\Tilde{f}(x,\nabla v_{\varepsilon}(x)) dx-\sum_{i=1}^{n+1}\alpha^{x_j}_i 
    \int_{\omega^j}\Tilde{f}(x,\xi^{x_j}_i)dx\Big{|}\\
&+\Big{|}\sum_{i=1}^{n+1}\alpha^{x_j}_i \int_{\omega^j}\Tilde{f}(x,\xi^{x_j}_i)dx-
   \int_{\omega^j} \Tilde{f}^{**}(x,\nabla u(x))dx\Big{|}\\ 
   \le&\Big{|}\sum_{i=1}^{n+1} 
   \left[\int_{\omega^j_i}\Tilde{f}(x,\xi^{x_j}_i)dx-\alpha^{x_j}_i 
    \int_{\omega^j}\Tilde{f}(x,\xi^{x_j}_i)dx\right]\Big{|}\\
&+\int_{\omega^j\setminus\bigcup_{i}\omega_i^j}\Tilde{f}(x,\nabla v_j(x))dx\\
&+\Big{|}\sum_{i=1}^{n+1}\alpha^{x_j}_i \int_{\omega^j}\Tilde{f}(x,\xi^{x_j}_i)dx-
   \int_{\omega^j} \Tilde{f}^{**}(x,\nabla u(x))dx\Big{|}
   \\
   &\leq
   \frac{\varepsilon|\omega^j|}{3|\Omega|}
+2\int_{\omega^j\setminus T}a(x)+\int_{\bigcup_i\omega^j_i\setminus T}a(x)+
\int_{\omega^j\setminus\bigcup_i \omega_i^j}a(x)dx+\frac{\varepsilon|\omega^j|}{3|\Omega|}\,.
\label{stima finale j}
\end{align}
We recall by (\ref{stima a T})
\begin{equation}
    \int_{\Omega\setminus T}a(x)dx<\frac{\varepsilon}{24}\,
\end{equation}
and furthermore by (\ref{differenza omega})
\begin{equation}
    \big|\cup_j(\omega^j\setminus\cup_i\omega_i^j) \big|=\sum_j(|\omega^j|-\sum_i|\omega_i^j|)\leq \delta
\end{equation}
so that also in this case
\begin{equation}
    \int_{\bigcup_j(\omega^j\setminus\bigcup_i \omega_i^j)}a(x)dx<\frac{\varepsilon}{24}\,.
\end{equation}
Since by definition $\omega^j$ are mutually disjoints we can conclude recalling the estimate \eqref{stima A} which become
\begin{align}
\Big{|}\int_{\Omega}&\Tilde{f}(x,\nabla v_{\varepsilon}(x)) dx-\int_{\Omega} \Tilde{f}^{**}(x,\nabla u(x))dx  \Big{|}
      \\&\leq
      \int_{\Omega\setminus T}a(x)dx
     +\sum_{j=1}^\infty
     \Big{|}\int_{\omega^j}\Tilde{f}(x,\nabla v_{\varepsilon}(x)) dx-\int_{\omega^j} \Tilde{f}^{**}(x,\nabla u(x))dx  \Big{|}
     \\&<\frac{\varepsilon}{24}+\frac{2|\Omega|\varepsilon}{3|\Omega|}+\frac{\varepsilon}{6}\,.   
\end{align}
i.e.
\begin{equation}
   \Big{|}\int_{\Omega}\Tilde{f}(x,\nabla v_{\varepsilon}(x)) dx-\int_{\Omega} \Tilde{f}^{**}(x,\nabla u(x))dx  \Big{|}<\varepsilon\,.
\end{equation}
\vspace{\baselineskip}

\textit{Step 4.}
 Now we consider the case where $u$ is piecewise affine,  so that we can split $\Omega$, minus a negligible set $N'$, in an union of disjoint Lipschitz open sets $\{\Omega_d\}_{d}$, with $1\leq d\leq M$, where $u$ is affine. Thanks to the previous steps, for every $d\in\{1,\dots,M\}$, we can find a function $v^d$ such that
\begin{equation}
    v^d\in u+W^{1,\infty}_0(\Omega_d)\,,\quad \|v^d-u\|_{\infty}<\varepsilon\,,\quad \|\nabla v^d\|_{\infty}<K
\end{equation} 
and
\begin{equation}
   \Big{|}\int_{\Omega_d}\Tilde{f}(x,\nabla v^d(x)) dx-\int_{\Omega_q} \Tilde{f}^{**}(x,\nabla u(x))dx  \Big{|}<\varepsilon\frac{|\Omega_q|}{|\Omega|}\,.
\end{equation}
So we define
\begin{equation}
    v_{\varepsilon}(x):=\begin{cases}
        v^d(x)\quad\text{in}\quad \Omega_d
        \\
        u(x)\quad\text{in}\quad N'
    \end{cases}
\end{equation}
and we have 
\begin{equation}
    v_{\varepsilon}\in u+W^{1,\infty}_0(\Omega)\,,\quad \|v_{\varepsilon}-u\|_{\infty}<\varepsilon\,,\quad \|\nabla v_{\varepsilon}\|_{\infty}<K
\end{equation} 
and
\begin{equation}
   \Big{|}\int_{\Omega}\Tilde{f}(x,\nabla v_{\varepsilon}(x)) dx-\int_{\Omega} \Tilde{f}^{**}(x,\nabla u(x))dx  \Big{|}<\varepsilon\,.
\end{equation}
\vspace{\baselineskip}

\textit{Step 5.}
In the general case if $u\in W^{1,\infty}(\Omega)$ by \cite[Chapter X, Proposition 2.9, page 317]{ET} then there exist a sequence of Lipschitz open sets $\Omega_{l}\subset \Omega$ and a sequence $u_l\in u+W_0^{1,\infty}(\Omega)$ such that $u_l$ is piecewise affine over $\Omega_l$ and
\begin{equation}
    |\Omega\setminus\Omega_l|\to 0\,,\quad \quad \|\nabla u_l\|_{\infty}\leq \|\nabla u\|_{\infty}+c(l) \quad\text{where }c(l)\to 0\,,
\end{equation}
\begin{equation}
    \|u_l-u\|_{\infty}\to 0\,, \quad \quad \nabla u_l\to\nabla u\quad \text{a.e. in }\Omega\,.
\end{equation}
Since $\Tilde{f}^{**}(x,\cdot)$ is continuous for every $x$, then $\Tilde{f}^{**}(x,\nabla u_l(x))$ converges a.e. to $\Tilde{f}^{**}(x,\nabla u(x))$.
 So, by assumption $b)$ and Lebesgue dominated convergence Theorem, we obtain
\begin{equation}
    \lim_{l\to +\infty}\int_{\Omega}|\Tilde{f}^{**}(x,\nabla u_l(x))-\Tilde{f}^{**}(x,\nabla{u}(x))|dx=0\,.
\end{equation}
So it exists $\Tilde{u}\in u+W_0^{1,\infty}(\Omega)$ and $\Tilde{\Omega}\subseteq \Omega$ such that $\Tilde{u}$ is piecewise affine over $\Tilde{\Omega}$,
\begin{equation}
    \int_{\Omega\setminus\Tilde{\Omega}}a(x)<\frac{\varepsilon}{4}\,,
    \quad \|\nabla\Tilde{u}\|_{\infty}<K\,,
    \quad \|\Tilde{u}-u\|_{\infty}<\frac{\varepsilon}{2}\,,
\end{equation}
and
\begin{equation}
    \int_{\Omega}|\Tilde{f}^{**}(x,\nabla\Tilde{u}(x))-\Tilde{f}^{**}(x,\nabla{u}(x))|dx<\frac{\varepsilon}{4}\,.
\end{equation}
Now, using the results of the previous steps, we can find $\Tilde{v}\in \Tilde{u}+W_0^{1,\infty}(\Tilde{\Omega})$
such that
\begin{equation}
    \|\nabla\Tilde{v}\|_{\infty}<K\,,
    \quad
    \|\Tilde{v}-\Tilde{u}\|_{\infty}<\frac{\varepsilon}{2}
\end{equation}
and
\begin{equation}
 \Big|\int_{\Tilde{\Omega}}\Tilde{f}(x,\nabla\Tilde{v}(x))-\Tilde{f}^{**}(x,\nabla{\Tilde{u}}(x))dx\Big|<\frac{\varepsilon}{4}\,.
\end{equation}
So taking 
\begin{equation}
    v_{\varepsilon}(x):=
\begin{cases}
    \Tilde{v}(x)\quad \text{in}\quad \Tilde{\Omega}
    \\
    \Tilde{u}(x)\quad \text{in}\quad \Omega\setminus\Tilde{\Omega}
\end{cases}  
\end{equation}
then we have 
\begin{equation}
    v_{\varepsilon}\in u+W^{1,\infty}(\Omega)\,,\quad \|v_{\varepsilon}-u\|_{\infty}<\varepsilon\,,\quad \|\nabla v_{\varepsilon}\|_{\infty}<K
\end{equation} 
and
\begin{equation}
   \Big{|}\int_{\Omega}\Tilde{f}(x,\nabla v_{\varepsilon}(x)) dx-\int_{\Omega} \Tilde{f}^{**}(x,\nabla u(x))dx  \Big{|}<\varepsilon\,.
\end{equation}

{\it Step 6.} Now we want to prove the second part of the statement about the integral representation of $sc^-(F_{\infty})$. By previous steps for every $u\in W^{1,\infty}(\Omega)$ there exists a sequence $u_n\in u+W^{1,\infty}_0(\Omega)$ such that
\begin{equation}
    \|\nabla u_n\|_{\infty}\leq K, \quad \|u_n-u\|_{\infty}\to 0
\end{equation}
and
\begin{equation}
    \lim_{n\to \infty}\int_{\Omega}\Tilde{f}(x,\nabla u_n(x))dx=\int_{\Omega}\Tilde{f}^{**}(x,\nabla u(x))dx\,.
\end{equation}
Up to a subsequence, we can suppose that $u_n\overset*\rightharpoonup u$, i.e. $u_n$ weakly$^*$ converges to $u$ in $W^{1,\infty}$
and we know, by a generalization of Tonelli theorem about lower semicontinuity of integral functionals with convex Lagrangian (cfr \cite[Chapter VIII, Theorem 2.1, page 243]{ET}), that
\begin{equation}
    \int_{\Omega}\Tilde{f}^{**}(x,\nabla u(x))dx\leq
    \liminf_{n\to \infty}\int_{\Omega}\Tilde{f}(x,\nabla v_n(x))dx
\end{equation}
for every $v_n\in  u+W^{1,\infty}_0(\Omega)$.\\
Thus we have that
\begin{multline}
\inf\Bigg{\{}\liminf\int_{\Omega}\Tilde{f}(x,\nabla v_n)dx\Bigg{|}
       v_n{\overset{*}{\rightharpoonup}}{u}, v_n\in u+W^{1,\infty}_0(\Omega)
    \Bigg{\}}\\
    =\int_{\Omega} \Tilde{f}^{**}(x,\nabla{u}(x))dx\,.
\end{multline}
i.e.
\begin{equation}
    sc^-(F_{\infty})(u)=\int_{\Omega} \Tilde{f}^{**}(x,\nabla{u}(x))dx
\end{equation}
where $sc^-(F_{\infty})$ is the lower semicontinuous envelope respect to the weak$^*$ topology of $W^{1,\infty}(\Omega)$.

\end{proof}

\begin{remark} We want to discuss  assumption $c)$ of Hypothesis \ref{Ipo caso x}. First of all we notice that assumption $a)$ and Lusin's Theorem would imply that for every $\xi$ and for every $\delta>0$  there exists $T_{\xi }\subset \Omega$ compact such that $|\Omega\setminus T_{\xi}|<\delta$ and $\Tilde{f}(\cdot,\xi)_{|T_{\xi}}$ is continuous.
In other words in our hypothesis c) we state a property that requires that Lusin's Theorem  holds uniformly as $\xi$ varies in $B_K(0)$. 

In \cite{ET}[Theorem 3.3, page 332], the authors assumed that $f(x,\xi)$ is a function Carathéodory, then 
Scorza-Dragoni Theorem implies that our assumption $c)$ holds true.


Moreover we remark that the function $\Tilde{f}(x,\xi):=g(x)h(\xi)$, with $g(x)\in L^1(\Omega)$ and $h(\xi)$ borelian and bounded on bounded sets, satisfies our assumption but are not Carath\'eodory. 
\end{remark}


\section{An approximation result for the general case}
\label{sezionegenerale}
In this section we want to generalize the validity of the integral representation formula (\ref{relaxed}) with $p=q=\infty$ proved in Theorem \ref{lemma} to a general Lagrangian $f(x,u,\xi):\Omega\times \R\times \R^N\to \R$ which satisfies following hypotheses.   

\begin{assumption}
   \label{gen ipo}
    The function  $f(x,u,\xi):\Omega\times \R\times \R^N\to \R$ satisfies
\begin{itemize}
    \item[a)] $f(x,u,\xi):\Omega\times\R\times\R^N\to\R$ is a borelian function;
    \item [b)] for every bounded set of $B\subset \R\times \R^N$ there exists $a\in L^1(\Omega)$ such that $|f(x,u,\xi)|\leq a(x)$ for all $(u,\xi)\in B$,
    \item[c)] for every $u\in  W^{1,\infty}(\Omega)$,  for every $\Tilde B\subset\R^N$ bounded set and for every $\delta>0$  there exists $T\subset \Omega$ compact such that $|\Omega\setminus T|<\delta$ and $f(x,u(x),\xi)$ is
continuous with respect to $x\in T$ uniformly as $\xi$ varies in $\Tilde B$,
    \item[d)] for almost every $x$  the function $f(x,u,\xi)$ is continuous with respect to $u$ uniformly as $\xi$ varies in bounded sets.
\end{itemize} 
\end{assumption} 
We introduce the following auxiliary functions that will be useful in the next proposition and in the main theorem of this section. For every function $f:\Omega\times\R\times\R^N\to\R$ and for every $K$ in $\N$ we define 
 $f_K$ as
\begin{equation}
\label{def fK}
    f_K(x,u,\xi)=\begin{cases}
    & f(x,u,\xi)\quad\quad $\text{if}$ \quad \|\xi\|\leq K\,,\\
    & +\infty \quad\quad\quad\quad $\text{if}$ \quad \|\xi\|> K\,.
   \end{cases}
\end{equation}
We remark that  $(f_K)^{**}(x,u,\xi)\ge(f^{**})_K(x,u,\xi)$ for every $(x,u,\xi)\in\R\times\R^N$ and in general the strict inequality may hold.
Moreover it is straightforward that
\begin{equation}
    (f_K)^{**}(x,u,\xi)\geq (f_{K+1})^{**}(x,u,\xi)
\end{equation}
for every $(x,u,\xi)\in\Omega\times\R\times\R^N$ and in particular $(f_K)^{**}$ is a monotone decreasing sequence pointwise convergent to $f^{**}$.

\begin{proposition}
\label{prop K gen}
Let  $\Omega$ be an open bounded Lipschitz subset of $\R^N$ and let $f:\Omega\times\R\times\R^N\to\R$ satisfy Hypotheses \ref{gen ipo}.\\
For every $\overline{u}\in W^{1,\infty}(\Omega)$ and for every $K\in\N$ such that
\begin{equation}
    \|\overline{u}\|_{\infty}+\|\nabla \overline{u}\|_{\infty}<K
\end{equation}
there exists a sequence $u_{K,n}\in\overline{u}+W_0^{1,\infty}(\Omega)$ such that
\begin{equation}
  \|\nabla u_{K,n}\|_{\infty}<K\,,\quad \quad   u_{K,n}\overset{*}{\rightharpoonup}\overline{u}
\end{equation}
and
\begin{equation}
    \lim_{n\to \infty}\Bigg{|}\int_{\Omega}f(x,u_{K,n}(x),\nabla u_{K,n}(x))dx-\int_{\Omega} (f_K)^{**}(x,\overline{u}(x),\nabla\overline{u}(x))dx\Bigg{|} =0\,.
\end{equation}
\end{proposition}

\begin{proof}

By Hypothesis \ref{gen ipo} $b)$ we fix $\overline{u}\in W^{1,\infty}(\Omega)$ and we can choose $K\in\N$ and $a_K\in L^1(\Omega)$  such that 
\begin{equation}\label{kubar}
    \|\overline{u}\|_{\infty}+\|\nabla \overline{u}\|_{\infty}<K
\end{equation}
and
\begin{equation}
\label{leb}
    |f(x,\overline u(x),\xi)|\leq a_K(x)
\end{equation}
for every $\xi\in B_K(0)$.

\noindent We define $\Tilde{f}_K:\Omega\times B_K\to\R$  as 
\begin{equation}
    \Tilde{f}_K(x,\xi):=f(x,\overline{u}(x),\xi)+a_K(x)\,.
\end{equation}
We observe that $\Tilde{f}_K$ satisfies Hypothesis \ref{Ipo caso x}, in fact 
\begin{itemize}
    \item $\Tilde{f}_K(x,\xi):\Omega\times B_K\to\R$ is borelian,
    \item $0\leq \Tilde{f}_K(x,\xi)\leq 2a_K(x)$ for all $(x,\xi)\in\Omega\times B_K(0)$,
    \item for every $\delta>0$ then there exists a compact set $T\subset \Omega$  such that $|\Omega\setminus T|<\delta$ and $\Tilde{f}_K(x,\xi)$ is
continuous with respect to $x\in T$ uniformly as $\xi$ varies in $B_K(0)$\,.
\end{itemize}
We can apply Theorem \ref{lemma} to $\Tilde{f}_K$ and 
thus for every $n\in\N$ we can find $u_{K,n}$ in $\overline{u}+ W^{1,\infty}_0(\Omega)$ such that
\begin{equation}\
    \|\nabla u_{K,n}\|_{\infty}<K \,,
\quad\quad \|u_{K,n}-\overline{u}\|_{\infty}<\frac{1}{n}
\end{equation}
and 
\begin{equation}\label{kueps}
    \Big{|}\int_{\Omega}\Tilde{f}_K(x,\nabla u_{K,n}(x))dx        
    -
    \int_{\Omega} (\Tilde{f}_K)^{**}(x,\nabla\overline{u}(x))dx\Big{|}<\frac{1}{n}
\end{equation}
i.e.
\begin{equation}
    \Big{|}\int_{\Omega}f(x,\overline{u}(x),\nabla u_{K,n}(x))dx        
    -
    \int_{\Omega} (f_{K})^{**}(x,\overline{u}(x),\nabla\overline{u}(x))dx\Big{|}<\frac{1}{n}\,.
\end{equation}
Now, in order to prove our propositions, it is sufficient to show that
\begin{equation}
\label{cont in u}
    \lim_{n\to +\infty}\Big{|}\int_{\Omega}f(x,u_{K,n}(x),\nabla u_{K,n}(x))dx        
    -
    \int_{\Omega} f(x,\overline{u}(x),\nabla u_{K,n}(x))dx\Big{|}=0\,.
\end{equation}
We recall that, by \eqref{kubar} and \eqref{kueps}, there exists $\Tilde{K}$ in $\N$ such that 
\begin{equation}
    \|u_{K,n}\|_{\infty}<\Tilde{K}
\end{equation}
and so 
\begin{equation}
    (u_{K,n}(\cdot),\nabla u_{K,n}(\cdot))\in [-\Tilde{K},\Tilde{K} ]\times B_{K}(0) \quad \text{a.e. in} \quad \Omega\,.
\end{equation}
We recall that Hypothesis \ref{gen ipo} d) implies that  $f(x,u,\xi)$ is uniformly continuous on $u$ in $ [-\Tilde{K},\Tilde{K}]$ uniformly as $\xi$ varies in $B_{K}(0)$ for almost every $x$ and we observe that, from this fact and the uniform convergence of $u_{K,n}$ to $\overline{u}$ in $\Omega$, it follows 
    \begin{equation}
f(x,u_{K,n}(x),\nabla u_{K,n}(x))-f(x,\overline{u}(x),\nabla u_{K,n}(x))\to 0
\end{equation} for a.e. $x\in\Omega$.
\\
By Hypotheses \ref{gen ipo} $b)$ there exists $a_{\Tilde{K}}\in L^1(\Omega)$ that 
\begin{equation}
    |f(x,u,\xi)|\leq a_{\Tilde{K}}(x)
\end{equation}
in $\Omega\times[-\Tilde{K},\Tilde{K} ]\times B_K(0)$ and so by Lebesgue dominated convergence theorem 
\begin{equation}
    \lim_{n\to +\infty}\Big{|}\int_{\Omega}f(x,u_{K,n}(x),\nabla u_{K,n}(x))dx        
    -
    \int_{\Omega} f(x,\overline{u}(x),\nabla u_{K, n}(x))dx\Big{|}=0\,.
\end{equation}
Therefore, passing to a suitable subsequence, we get the conclusion.

\end{proof}



Now we are ready to prove the main result of this section.

\begin{theorem}
\label{pro gen}
    Let  $\Omega$ be an open bounded Lipschitz subset of $\R^N$ and let $f:\Omega\times\R\times\R^N\to\R$ satisfy Hypotheses \ref{gen ipo}.\\
For every $\overline{u}\in W^{1,\infty}(\Omega)$ exists a sequence $u_{n}\in\overline{u}+W_0^{1,\infty}(\Omega)$ such that
\begin{equation}
    \lim_{n\to \infty}\|u_{n}-\overline{u}\|_{\infty}=0
  \end{equation}
and
\begin{equation}
    \lim_{n\to \infty}\int_{\Omega}f(x,u_{n}(x),\nabla u_{n}(x))dx
 =
    \int_{\Omega} f^{**}(x,\overline{u}(x),\nabla\overline{u}(x))dx\,.
\end{equation}
Furthermore
\begin{equation}
    sc^-(F_{\infty})(\overline{u})
    =\int_{\Omega} f^{**}(x,\overline{u}(x),\nabla\overline{u}(x))dx
    \end{equation}
where with $sc^-(F_{\infty})$ we denote the lower semicontinuous envelope of 
\begin{equation}
    F_{\infty}(u)=\int_{\Omega}{f}(x,u,\nabla u(x))dx
\end{equation}
with respect to the weak$^*$ topology of $W^{1,\infty}(\Omega)$. 
\end{theorem}
\begin{proof}
    
For every $K$ in $\N$ we consider the function $f_K(x,u,\xi)$ defined in \eqref{def fK} and its lower semi-continuous convex envelope  w.r.t. $\xi$ that, as usual, we denote by $(f_K)^{**}(x,u,\xi)$. We remark that
\begin{equation}
    (f_K)^{**}(x,u,\xi)\geq (f_{K+1})^{**}(x,u,\xi)
\end{equation}
for every $(u,\xi)\in\R\times\R^N$ and in particular $(f_K)^{**}$ is a monotone decreasing sequence pointwise convergent to $f^{**}$.
\\
Fixed  $\overline{u}\in W^{1,\infty}(\Omega)$ we have $(f_K)^{**}(\cdot,\overline{u}(\cdot),\nabla \overline{u}(\cdot))$ is a decreasing sequence pointwise a.e. convergent to $f^{**}(\cdot,\overline{u}(\cdot),\nabla \overline{u}(\cdot))$.
\\
Let $\overline K$ such that 
\begin{equation}
    \|\overline{u}\|_{\infty}+\|\nabla \overline{u}\|_{\infty}<\overline K\,.
\end{equation} 
By Hypothesis \ref{gen ipo} b), there exists $a_{\overline K}\in L^1(\Omega)$  such that for every $K\geq \overline{K}$ and for a.e. $x\in\Omega$
\begin{equation}
    (f_K)^{**}(x,\overline{u}(x),\nabla\overline{u}(x))
    \leq (f_K)(x,\overline{u}(x),\nabla\overline{u}(x))
    =f(x,\overline{u}(x),\nabla\overline{u}(x))
    \leq a_{\overline K}(x)\,.
\end{equation}
Monotone convergence theorem implies that
\begin{multline}
     \lim_{K\to +\infty}\int_{\Omega} a_{\overline K}(x)-(f_K)^{**}(x,\overline{u}(x),\nabla\overline{u}(x))dx
    \\=\int_{\Omega} a_{\overline K}(x)-f^{**}(x,\overline{u}(x),\nabla\overline{u}(x))dx\,
\end{multline}
and also
\begin{equation}
    \lim_{K\to +\infty}\int_{\Omega}(f_K)^{**}(x,\overline{u}(x),\nabla\overline{u}(x))dx
    =\int_{\Omega}f^{**}(x,\overline{u}(x),\nabla\overline{u}(x))dx\,.
\end{equation}
By Proposition \ref{prop K gen}, for every $\overline{u}\in W^{1,\infty}(\Omega)$ and for every $K\in \N$ such that
\begin{equation}
    \|\overline{u}\|_{\infty}+\|\nabla\overline{u}\|_{\infty}<K
\end{equation} 
there exists a sequence $u_{K,n}\in \overline{u}+W^{1,\infty}_0(\Omega)$ such that
\begin{equation}
    \|\nabla u_{K,n}\|<K\,,\quad  u_{K,n}\overset{*}{\rightharpoonup}\overline{u}
\end{equation}
and
\begin{equation}
    \lim_{n\to \infty}\Big{|}\int_{\Omega}f(x,u_{K,n}(x),\nabla u_{K,n}(x))dx        
    -
    \int_{\Omega} (f_{K})^{**}(x,\overline{u}(x),\nabla\overline{u}(x))dx\Big{|}=0\,.
\end{equation}
Taking the double limit
\begin{equation}
    \lim_{K\to +\infty}\Bigg{(}\lim_{n\to \infty}\int_{\Omega}f(x,u_{K,n}(x),\nabla u_{K,n}(x))dx\Bigg{)}
    =\int_{\Omega} f^{**}(x,\overline{u}(x),\nabla\overline{u}(x))dx\,.
\end{equation}
This implies the existence of a sequence $\overline{u}_{n}$ in $\overline{u}+W^{1,\infty}_0(\Omega)$ such that
\begin{equation}
    \quad\quad\lim_{n\to \infty}\|\overline{u}_{n}-\overline{u}\|_{\infty}=0\,,
\end{equation}
\begin{equation}
\lim_{n\to \infty}\int_{\Omega}f(x,\overline{u}_{n}(x),\nabla \overline{u}_{n}(x))dx
=\int_{\Omega} f^{**}(x,\overline{u}(x),\nabla\overline{u}(x))dx\,.
\end{equation}
Furthermore for every $K$
\begin{multline}\label{min}
     \inf\Bigg{\{}\liminf\int_{\Omega}f(x,u_n,\nabla u_n)dx\Bigg{|}
       u_n{\overset{*}{\rightharpoonup}}\,\overline{u},\,
     \|\nabla u_{n}\|_{\infty}<K
    \Bigg{\}}\\
    \leq\int_{\Omega} (f_K)^{**}(x,\overline{u}(x),\nabla\overline{u}(x))dx\,.
\end{multline}
Since $(f_K)^{**}(x,u,\xi)$ is convex and l.s.c. with respect to $\xi$ we have (for example by \cite[Chapter VIII, Theorem 2.1, page 243]{ET}) that the functional 
\begin{equation}
    \int_{\Omega}(f_K)^{**}(x,u(x),\nabla u(x))dx
\end{equation}
is sequentially lower semicontinous with respect to the weak topology of $W^{1,1}(\Omega)$. So a fortiori
\begin{multline}\label{magg}
     \inf\Bigg{\{}\liminf\int_{\Omega}f(x,u_n,\nabla u_n)dx\Bigg{|}
       u_n{\overset{*}{\rightharpoonup}}\,\overline{u},\,
     \|\nabla u_{n}\|_{\infty}<K
    \Bigg{\}}\\
    \geq\int_{\Omega} (f_K)^{**}(x,\overline{u}(x),\nabla\overline{u}(x))dx\,
\end{multline}
and then \eqref{min} and \eqref{magg} imply that the equality holds
\begin{multline}
\inf\Bigg{\{}\liminf\int_{\Omega}f(x,u_n,\nabla u_n)dx\Bigg{|}
       u_n{\overset{*}{\rightharpoonup}}\,\overline{u},\,
     \|\nabla u_{n}\|_{\infty}<K
    \Bigg{\}}\\
    =\int_{\Omega} (f_K)^{**}(x,\overline{u}(x),\nabla\overline{u}(x))dx\,.
\end{multline}
Finally we have that, for $K\to +\infty$, 
\begin{equation}
    \inf\Bigg{\{}\liminf\int_{\Omega}f(x,u_n,\nabla u_n)dx\Bigg{|}u_n{\overset{*}{\rightharpoonup}}\,\overline{u}\Bigg{\}}=\int_{\Omega} f^{**}(x,\overline{u}(x),\nabla\overline{u}(x))dx\,.
\end{equation}
i.e.
\begin{equation}
    sc^-(F_{\infty}(\overline{u}))
    =\int_{\Omega} f^{**}(x,\overline{u}(x),\nabla\overline{u}(x))dx\,.
    \end{equation}
\end{proof}

\begin{remark}
We notice that Hypothesis \ref{gen ipo} d) on the continuity of the lagrangian w.r.t. $u$ is strongly used in the proof of our result. The same assumption is also present in previous papers \cite{ET} and \cite{MS}. 
\end{remark}

\begin{remark}
We point out the fact that, as a corollary, we can replace Hypotheses \ref{gen ipo} c) and d) with the more restrictive  request that for every bounded set $\Tilde B\subset\R^N$  and for every $\delta>0$  there exists a compact set $T\subset \Omega$  such that $|\Omega\setminus T|<\delta$ and $f(x,u,\xi)$ is
continuous with respect to $(x,u)\in T\times\R$ uniformly as $\xi$ varies in $\Tilde B$.


\end{remark}

\begin{remark} Proposition \ref{prop K gen} can be seen as a generalization of Theorem 3.7 in \cite[Chapter X]{ET}, at least for the case of Lipschitz functions. In particular we underline that in \cite{ET} the authors assume that the lagrangian is a Carathéodory function, continuous w.r.t. $(u,\xi)$.

\end{remark}

\begin{remark}
In \cite{MS} is presented a relaxation result under the assumption that $f(x,u,\xi)$ is  continuous on $u$ uniformly as $\xi$ varies in a bounded set of $\R^N$ and $f$ is also assumed to be upper semicontinuous with respect to $\xi$. 
Thus $f(x,u,\xi)$ is upper semicontinuous with respect to $(u,\xi)$. In the next examples we show that our Hypotheses \ref{gen ipo} c) and d) are neither more general nor more restrictive than the one presented in \cite{MS}.
\end{remark}

\begin{example}



We consider the set $\Omega:=]-\frac{\pi}{2},\frac{\pi}{2}[$ and the function $f:]-\frac{\pi}{2},\frac{\pi}{2}[\times \R\to\R$ defined by

\begin{equation}
    f(x,\xi):=
\begin{cases}
        0\quad\text{if}\quad \xi>\tan(x)
        \\
        1\quad\text{if}\quad \xi\leq\tan(x)\,.
    \end{cases}
\end{equation}
It is easy to check that $f(x,\xi)$ is upper semicontinuous on $\Omega\times\R$. On the other hand, given $M\in\R$ and $\delta<2\arctan M$, we have that,  for every $T\subset]-\frac{\pi}{2},\frac{\pi}{2}[$ such that $|]-\frac{\pi}{2},\frac{\pi}{2}[\setminus T|<\delta$, there exists $\bar x\in T$ such that $f$ is not continuous at $(\bar x, \tan \bar x)$.  This shows that the assumptions in \cite{MS} do not imply our hypotheses.
\end{example}

\begin{example}
On the other side
if $f(x,u,\xi):=g(x,u)h(\xi)$ with
\begin{equation}
    g(x,u):=\|x\|+|u|
\end{equation}
and
\begin{equation}
    h(\xi):=\begin{cases}
        1\quad \text{if} \quad \xi\in\mathbb{Q} \\
        -1 \quad \text{if} \quad \xi\in \R \setminus\mathbb{Q}
    \end{cases}
\end{equation}  then $f(x,u,\xi)$ satisfies general hypotheses but it is not upper semicontinuous with respect to $\xi$.
\end{example} 

\section{Lavrentiev Phenomenon}
\label{sezioneLav}
In this section we apply the results of previous sections to show that whenever the Lavrentiev phenomenon does not occur for the functional 
\begin{equation}
\int_{\Omega}f^{**}(x,u(x),\nabla u(x)) dx
\end{equation}
then
it  does not occur also for the functional
\begin{equation}
\int_{\Omega}f(x,u(x),\nabla u(x)) dx.
\end{equation}
\begin{theorem}
\label{theo Lavr}
    Let $\Omega$ be an open bounded Lipschitz subset of $\R^N$, let  $f:\Omega\times\R\times\R^N\rightarrow \R$ satisfy Hypotheses \ref{gen ipo}, $\varphi\in  W^{1,\infty}(\Omega)$ and $1\leq p<+\infty$. 
    
    If
    \begin{equation}
    \label{Ipo gen}
\inf_{\varphi+W^{1,p}_0(\Omega)}\int_{\Omega}f^{**}(x,u(x),\nabla u(x))dx=\inf_{\varphi+W^{1,\infty}_0(\Omega)}\int_{\Omega}f^{**}(x,u(x),\nabla u(x))dx
\end{equation}
then
\begin{equation}
\inf_{\varphi+W^{1,p}_0(\Omega)}\int_{\Omega}f(x,u(x),\nabla u(x))dx=\inf_{\varphi+W^{1,\infty}_0(\Omega)}\int_{\Omega}f(x,u(x),\nabla u(x))dx\,.
\end{equation}  
\end{theorem}
\begin{proof}

The properties of $f$ and $f^{**}$ imply that,
for every $\varphi\in W^{1,\infty}(\Omega)$, 
\begin{equation}
\label{con p gen}
    \inf_{\varphi+W^{1,q}_0(\Omega)}\int_{\Omega}f^{**}(x,u(x),\nabla u(x))dx
    \leq
 \inf_{\varphi+W^{1,q}_0(\Omega)}\int_{\Omega}f(x,u(x),\nabla u(x))dx\
\end{equation}
for every $1\leq q \le \infty$. 
We have also that
\begin{equation}
\label{ine f gen}
    \inf_{\varphi+W^{1,p}_0(\Omega)}\int_{\Omega}f(x,u(x),\nabla u(x))dx
    \leq
    \inf_{\varphi+W^{1,\infty}_0(\Omega)}\int_{\Omega}f(x,u(x),\nabla u(x))dx\
\end{equation}
and, using both hypothesis (\ref{Ipo gen}) and (\ref{con p gen}) we get
\begin{equation}
    \inf_{\varphi+W^{1,\infty}_0(\Omega)}\int_{\Omega}f^{**}(x,u(x),\nabla u(x))dx
    \leq
    \inf_{\varphi+W^{1,p}_0(\Omega)}\int_{\Omega}f(x,u(x),\nabla u(x))dx\,.
\end{equation}
Furthermore Theorem \ref{pro gen} states that
\begin{equation}
     \inf_{\varphi+W^{1,\infty}_0(\Omega)}\int_{\Omega}f(x,u(x),\nabla u(x))dx
   =
    \inf_{\varphi+W^{1,\infty}_0(\Omega)}\int_{\Omega}f^{**}(x,u(x),\nabla u(x))dx
\end{equation}
and so 
\begin{multline}\label{lavr non conv}
    \inf_{\varphi+W^{1,\infty}_0(\Omega)}\int_{\Omega}f^{**}(x,u(x),\nabla u(x))dx
    \leq
    \inf_{\varphi+W^{1,p}_0(\Omega)}\int_{\Omega}f(x,u(x),\nabla u(x))dx\\
    \leq \inf_{\varphi+W^{1,\infty}_0(\Omega)}\int_{\Omega}f(x,u(x),\nabla u(x))dx
   =
    \inf_{\varphi+W^{1,\infty}_0(\Omega)}\int_{\Omega}f^{**}(x,u(x),\nabla u(x))dx\,.
\end{multline}
It follows that all the inequalities in \eqref{lavr non conv} are  equalities, proving the thesis.
    
\end{proof}


\section{Approximation results for the autonomous case}
\label{sezione auto}
Now we focus on autonomous Lagrangians, in order to apply some recent results about this case. As far as we know, also in the autonomous case our results were never proved before without assuming the continuity of Lagrangian respect to $\xi$. So we prefer report explicitly them also for Lagrangians which satisfy the following hypotheses which imply Hypothesis \ref{gen ipo} and that are quit natural. 
\begin{assumption}
\label{Bas Hyp}
The function $f:\R\times\R^N\rightarrow\R$ satisfies 
\begin{enumerate}[i)]
    \item $f(u,\xi):\R\times\R^N\rightarrow\R$ is borelian,
    \item
     $f(u,\xi)$ is continuous with respect to $u$ uniformly as $\xi$ varies on each bounded set of $\R^N$, 
\item $f(u,\cdot)$ is bounded on bounded sets of $\R^N$ for every $u\in \R$.
\end{enumerate}
\end{assumption}


Now we show that Hypotheses \ref{Bas Hyp} imply Hypotheses \ref{gen ipo}.
\begin{lemma}\label{ipotesi}
    Let  $\Omega$ be an open bounded Lipschitz subset of $\R^N$ and let $f:\R\times\R^N\to\R$ satisfy Hypothesis \ref{Bas Hyp}. Then $f$ satisfies Hypotheses \ref{gen ipo}. 
\end{lemma}
\begin{proof}
    Hypotheses \ref{gen ipo} $a)$ and $d)$ are immediately verified and to prove b) we can observe, arguing as in Lemma \ref{unif cont case x}, that 
$f(u,\xi)$ is uniformly continuous with respect to $u$ in bounded sets of $\R$ uniformly as $\xi$ varies in bounded sets of $\R^N$.
\\
Thus for every $u\in W^{1,\infty}(\Omega)$ and $B_K(0)\subset \R^N$ there exists a non decreasing modulus of continuity $\omega(\cdot)$ such that
 \begin{equation}
     |f(\overline{u}(x),\xi)-f(\overline{u}(y),\xi)|\leq
    \omega(|\overline{u}(x)-\overline{u}(y)|)
     \leq\omega(M|x-y|)
 \end{equation} 
 for every $\xi$ in $B_K(0)$. So $f(\overline{u}(x),\xi)$ is uniformly continuous for a.e. $x\in \Omega$ uniformly as $\xi$ varies in $B_K(0)$.
\\
Finally to prove the validity of Hypothesis \ref{gen ipo} $c)$ it is sufficient to show that $f$ is bounded on bounded sets. Given $(u,\xi)\in[-M,M]\times B_K(0)$, the function $f(u,\xi)$ is uniformly continuous with respect to $u\in[-M,M]$ uniformly as $\xi$ varies on $B_K(0)$.
Furthermore for every $u\in\R$ then $f(u,\cdot)$ is bounded on $B_K(0)$ by Hypothesis \ref{Bas Hyp}.
\\
So there exists a finite set $\{u_i| i=1,\dots,N\}\subset [-M,M]$   such that
\begin{equation}
\label{bound}
    |f(u,\xi)|\leq 1+\max_{1\leq i \leq N}\big{\{}|f(u_i,\xi)|\big{\}}\leq 1+\max_{1\leq i \leq N}\bigg{\{}\sup_{\xi\in B_K(0)}\big{\{} |f(u_i,\xi)| \big{\}}\bigg{\}}\in \R
\end{equation}
for every $u\in [-M,M]$ and for every $\xi\in B_K(0)$.

\end{proof} 


We are ready to state the results in the autonomous case. 

\begin{proposition}
    \label{prop K}
Let  $\Omega$ be an open bounded Lipschitz subset of $\R^N$ and let $f:\R\times\R^N\to\R$ satisfy Hypothesis \ref{Bas Hyp}.
\\
For every $\overline{u}\in W^{1,\infty}(\Omega)$ and for every $K\in\N$ such that
\begin{equation}
    \|\overline{u}\|_{\infty}+\|\nabla \overline{u}\|_{\infty}<K
\end{equation}
there exists a sequence $u_{K,n}\in\overline{u}+W_0^{1,\infty}(\Omega)$ such that
\begin{equation}
  \|\nabla u_{K,n}\|_{\infty}<K\,,\quad \quad   u_{K,n}\overset{*}{\rightharpoonup}\overline{u}\quad\text{in} \quad W^{1,\infty}(\Omega)
\end{equation}
and
\begin{equation}
    \lim_{n\to \infty}\Bigg{|}\int_{\Omega}f(u_{K,n}(x),\nabla u_{K,n}(x))dx-\int_{\Omega} (f_K)^{**}(\overline{u}(x),\nabla\overline{u}(x))dx\Bigg{|} =0\,.
\end{equation}
\end{proposition}

\begin{proof}In view of Lemma \ref{ipotesi} it is a immediate consequence of Proposition \ref{prop K gen}.

\end{proof}

\begin{theorem}
   \label{pro}
Let  $\Omega$ be an open bounded Lipschitz subset of $\R^N$ and let $f:\R\times\R^N\to\R$ satisfies Hypotheses \ref{Bas Hyp}.\\
For every $\overline{u}\in W^{1,\infty}(\Omega)$ exists a sequence $u_{n}\in\overline{u}+W_0^{1,\infty}(\Omega)$ such that
\begin{equation}
    \lim_{n\to \infty}\|u_{n}-\overline{u}\|_{\infty}=0
  \end{equation}
and
\begin{equation}
    \lim_{n\to \infty}\int_{\Omega}f(u_{n}(x),\nabla u_{n}(x))dx        
 =
    \int_{\Omega} f^{**}(\overline{u}(x),\nabla\overline{u}(x))dx\,.
\end{equation}
Furthermore
\begin{equation}
    sc^-(F_{\infty})(\overline{u})
    =\int_{\Omega} f^{**}(\overline{u}(x),\nabla\overline{u}(x))dx
\end{equation} 
where with $sc^-(F_{\infty})$ we denote the lower semicontinuous envelope of 
\begin{equation}
    F_{\infty}(u)=\int_{\Omega}{f}(u(x),\nabla u(x))dx
\end{equation}
with respect to the weak$^*$ topology of $W^{1,\infty}(\Omega)$.

\end{theorem}

\begin{proof}
As for the previous Proposition we have only to notice that it is an immediate application of Theorem \ref{pro gen}. 

\end{proof}

\begin{remark} The assumption that $f$ is bounded on bounded sets in particular implies that, for every $K>0$, $(f_K)^{**}$ is bounded from below.
 \end{remark}
 
  \begin{remark} We notice that in our setting it could happen that
 $f^{**}\equiv-\infty$. In this case Theorem \ref{pro} implies that there exists a sequence in $u_n\in W^{1,\infty}(\Omega)$ such that $u_n$ converges to $\bar u$ in $L^\infty$ and 
\begin{equation}
    \lim_{n\to\infty}\int_{\Omega}f(u_n(x),\nabla u_n(x))dx=-\infty\,.
\end{equation} 
\end{remark}

\begin{remark}
    If $f(x,u,\xi):= g(x,u)+h(u,\xi)$ where $g$ is Carathedory and $h$ satisfies Hypotheses \ref{Bas Hyp} then $f(x,u,\xi)  $ satisfies Hypotheses \ref{gen ipo} but does not necessarily satisfy hypotheses assumed on \cite{ET} or on \cite{MS}.
\end{remark}

\noindent Now we can apply the previous results about Lavrentiev Phenomenon for the autonomous case.

\begin{theorem}
  \label{theo}
Let $\Omega$ be an open bounded Lipschitz subset of $\R^N$, let  $f:\R\times\R^N\rightarrow \R$ satisfy Hypotheses \ref{Bas Hyp}, $\varphi\in  W^{1,\infty}(\Omega)$ and $1\leq p<\infty$. 

If
\begin{equation}\label{noLavrConv}
    \inf_{\varphi+W^{1,p}_0(\Omega)}\int_{\Omega}f^{**}(u(x),\nabla u(x))dx=\inf_{\varphi+W^{1,\infty}_0(\Omega)}\int_{\Omega}f^{**}(u(x),\nabla u(x))dx\,.
\end{equation}
then
\begin{equation}
    \inf_{\varphi+W^{1,p}_0(\Omega)}\int_{\Omega}f(u(x),\nabla u(x))dx=\inf_{\varphi+W^{1,\infty}_0(\Omega)}\int_{\Omega}f(u(x),\nabla u(x))dx\,.
\end{equation}  
\end{theorem}

\begin{proof}
    In view of Lemma \ref{ipotesi} it is a immediate consequence of Theorem \ref{theo Lavr}. 
    
\end{proof}
Now we recall a recent result by Pierre Bousquet (\cite[Theorem 1.1]{Bousquet2023}) and then we will use it to prove Theorem \ref{Theo fin}.
\begin{theorem}
\label{theo Pierre}
    Let $\Omega$ be an open bounded Lipschitz subset of $\R^N$ and $\varphi\in  W^{1,\infty}(\Omega)$. Assume $g:\R\times\R^N\rightarrow \R^+$ continuous in both variable and convex with respect to the last variable. For every $u\in W^{1,1}_{\varphi}(\Omega) $ there exists a sequence $(u_n)_n\in W^{1,\infty}_{\varphi}(\Omega)$ such that $u_n$ strongly converges to $u$ in $W^{1,1}(\Omega)$ and
\begin{equation}
     \int_{\Omega} g(u_n(x),\nabla u_n(x))dx\to \int_{\Omega} g(u(x),\nabla u(x))dx\,.
\end{equation}
\end{theorem}


\begin{theorem}
    Let $\Omega$ be an open bounded Lipschitz subset of $\R^N$. Let $\varphi\in  W^{1,\infty}(\Omega)$ and assume $f:\R\times\R^N\to\R^+$ satisfies Hypotheses \ref{Bas Hyp} and $f^{**}$ continuous in both variable, then
    \begin{equation}
       \inf_{u\in\varphi+ W^{1,1}_{0}(\Omega)} \int_{\Omega} f(u(x),\nabla u(x))dx= \inf_{u\in\varphi+ W^{1,\infty}_{0}(\Omega)}\int_{\Omega} f(u(x),\nabla u(x))dx\,.
\end{equation}
\end{theorem}
\begin{proof}We can apply Theorem \ref{theo Pierre} to $f^{**}$ and so we note that
\begin{equation}
       \inf_{u\in\varphi+ W^{1,1}_{0}(\Omega)} \int_{\Omega} f^{**}(u(x),\nabla u(x))dx= \inf_{u\in \varphi+W^{1,\infty}_{0}(\Omega)}\int_{\Omega} f^{**}(u(x),\nabla u(x))dx\,.
\end{equation}
Now the assumptions of Theorem \ref{theo} are satisfied.

\end{proof}


\begin{remark} In general, under Hypothesis \ref{Bas Hyp}, the function $(u,\xi)\rightarrow f^{**}(u,\xi)$ could be not continuous (cfr \cite[Example 3.9]{MS}). In \cite[Corollary 3.12]{MS} the authors proved that if
\begin{enumerate}
    \item[i)] $f(u,\xi)$ is continuous in $u$ uniformly with respect to $\xi\in \R^N$
\end{enumerate}
or
\begin{enumerate}
    \item[ii)] $f(u,\xi)\geq \lambda_1\|\xi\|^{\alpha}+\lambda_2$ with $\alpha>1$, $\lambda_1>0 $ and $\lambda_2\in \R$ 
\end{enumerate}
then $f^{**}(u,\xi)$ is continuous in both variables. 
\end{remark}

In the next lemma we prove that Hypothesis \ref{Bas Hyp} implies  that $(f_K)^{**}$ is continuous on $\R\times B_K(0)$ for every $K\in\N$ and that $f^{**}$ is a Borel function on $\R\times\R^N$.

\begin{lemma}
    \label{fstarcont} 
If $f:\R\times\R^n\to\R$ satisfies Hypothesis \ref{Bas Hyp} then $(f_K)^{**}$ is continuous on $\R\times B_K(0)$  for every $K\in\N$.
Furthermore $f^{**}$ is a borelian function on $\R\times \R^N$.
\end{lemma}
\begin{proof}

First of all we prove  that $(f_K)^{**}(u,\xi)$ is continuous w.r.t. $u$ uniformly as $\xi$ varies in $B_K(0)$.  We fix $u_0\in \R$ and Hypothesis \ref{Bas Hyp}, $ii)$ guarantees that for every $\varepsilon>0$ there exists $\delta>0$ such that if $|u-u_0|<\delta$ then
\begin{equation}
    |f(u,\xi)-f(u_0,\xi)|<\varepsilon \quad \forall\xi\in B_K(0)
\end{equation}
and so we have
\begin{equation}
\label{cont uni}
    (f_K)^{**}(u,\xi)-\varepsilon\leq f(u,\xi)-\varepsilon\leq f(u_0,\xi)\,.
\end{equation}
Now $(f_K)^{**}(u,\xi)-\varepsilon$ is convex in $\xi$ and and so 
\begin{equation}
    (f_K)^{**}(u,\xi)-\varepsilon\leq (f_K)^{**}(u_0,\xi)\,.
\end{equation}
Since  $f(\cdot,\xi)$ is uniformly continuous in $[-M,M]$ uniformly as $\xi$ varies in $B_K(0)$, thus we can change the roles of
 $u$ and $u_0$ in the previous inequalities and we obtain 
\begin{equation}
    |(f_K)^{**}(u,\xi)-(f_K)^{**}(u_0,\xi)|<\varepsilon \quad \forall\xi\in B_K(0)\,.
\end{equation} 
 Since $(f_K)^{**}(u,\cdot)$ is continuous in $B_K(0)$ for every $u$, for every sequence $(u_n,\xi_n)\in \R\times B_K(0)$ converging to $(u_0,\xi_0)\in \R\times B_K(0)$ such that 
\begin{align}
\label{cont2var}
\lim_{n}\,|(f_K)^{**}(u_n,\xi_n)&-(f_K)^{**}(u_0,\xi_0)|\\
     \leq &\lim_{n}|(f_K)^{**}(u_n,\xi_n)-(f_K)^{**}(u_0,\xi_n)|\\&+\lim_{n}|(f_K)^{**}(u_0,\xi_n)-(f_K)^{**}(u_0,\xi_0)|\\=&0   
\end{align}
i.e. $(f_K)^{**}$ is continuous on $\R\times B_K(0)$.

Now, recalling  that
\begin{equation}
    f^{**}(u,\xi)=\inf_K(f_K)^{**}(u,\xi)=\lim_K(f_K)^{**}(u,\xi), 
\end{equation}
the continuity of $(f_K)^{**}$  implies that $f^{**}$ is borelian on $\R\times\R^N$.

\end{proof}

 \begin{corollary}
 \label{corollario f_K}
       Let $\Omega$ be an open bounded Lipschitz subset of $\R^N$ and let $f:\R\times\R^N\to\R$ satisfies Hypothesis \ref{Bas Hyp}. If for every $K>0$ there exist $K'\geq K$ such that 
\begin{equation}
    (f_{K'})^{**}_{|\R\times B_K}=(f^{**})_{|\R\times B_K}
\end{equation}
then $f^ {**}$ is continuous in both variables.
    \end{corollary}
\begin{proof}
    By Lemma \ref{fstarcont} we know that  $(f_{K'})^{**}_{|\R\times B_K}$ is continuous in both variables. Thus every restriction of $f^{**}$ is continuous and so $f^{**}$ is globally continuous in both variables.
    
\end{proof}
Now we state a relaxation result that holds under the assumption of uniform superlinearity of the Lagrangian. 
\begin{theorem}
\label{Theo fin}
Let $\Omega$ be an open bounded Lipschitz subset of $\R^N$. Let $\varphi\in  W^{1,\infty}(\Omega)$ and assume $f:\R\times\R^N\to\R^+$ satisfies Hypothesis \ref{Bas Hyp}, be uniformly superlinear, i.e. there exist a superlinear function $\theta:\R^n\to \R$ such that 
 \begin{equation}
     f(u,\xi)\geq \theta (\xi)\quad \forall u\in \R
 \end{equation}
and $f^{**}$ be continuous in both variables.\\
Then for every $u\in \varphi+W^{1,1}_0(\Omega)$
\begin{equation}
    sc^-(\overline{F_1})(u)=\int_{\Omega}f^{**}(u(x),\nabla u(x))dx
\end{equation}
where $sc^-(\overline{F_1}$) is the lower semicontinuous envelope of 
\begin{equation}
    \overline{F_1}(u)=\begin{cases}
        \int_{\Omega}f(u(x),\nabla u(x))dx\quad \text{if}\quad u\in \varphi+W^{1,\infty}_0(\Omega)
        \\
        +\infty \quad \text{if}\quad u\in \varphi+W^{1,1}_0(\Omega)\setminus W^{1,\infty}_0(\Omega)
    \end{cases}
\end{equation}
with respect to the weak topology of $W^{1,1}(\Omega)$.
\end{theorem}
\begin{proof}
By theorem \ref{theo Pierre} for every $u\in{\varphi}+W^{1,1}_0(\Omega)$ there exists a sequence $(u_k)_k\subset {\varphi}+W^{1,\infty}_0(\Omega)$
such that
\begin{equation}
    u_k\to u\quad \text{strongly in } W^{1,1}(\Omega)
\end{equation}
and
\begin{equation}
    \lim_k\int_{\Omega}f^{**}(u_k(x),\nabla u_k(x))dx=
    \int_{\Omega}f^{**}(u(x),\nabla u(x))dx\,.
\end{equation}
By Theorem \ref{pro}, for every $u_k$ there exist a sequence $(u^{(n)}_k)_n\subset {\varphi}+ W^{1,\infty}_0(\Omega)$ such that
\begin{equation}
   u^{(n)}_k\to u_k\quad \text{in}\quad L^{\infty}(\Omega)  
\end{equation}
and 
\begin{equation}
    \lim_n\int_{\Omega}f(u^{(n)}_k(x),\nabla u^{(n)}_k(x))dx=
    \int_{\Omega}f^{**}(u_k(x),\nabla u_k(x))dx\,.
\end{equation}
Thus there exists a sequence $u_n\in {\varphi}+W^{1,\infty}_0(\Omega)$ such that 
\begin{equation}
    u_n\to_{L^1} u
\end{equation}
and
\begin{equation}
    \lim_n\int_{\Omega}f(u_n(x),\nabla u_n(x))dx=
    \int_{\Omega}f^{**}(u(x),\nabla u(x))dx\,.
\end{equation}
Since $f(u,\xi)$ is uniformly superlinear there exists $\theta$ superlinear such that
\begin{equation}
    \sup_n\int_{\Omega}\theta(\nabla u_n(x))dx\leq\sup_n \int_{\Omega} f(u_n(x),\nabla u_n(x))dx<+\infty
    \label{stima theta}
\end{equation}
thus, unless considering a subsequence, $\nabla u_n$ converges weakly in $L^1$ to some $v\in L^1$. It is easy to see (using the integration by parts) that actually $v=\nabla u$.
\\
So we have 

\begin{equation}
    u_n\rightharpoonup_{W^{1,1}} u\,.
\end{equation}
Since by \cite[Chapter VIII, Theorem 2.1, page 243]{ET} 
\begin{equation}
    \int_{\Omega}f^{**}(u(x),\nabla u(x))dx\leq 
    \int_{\Omega}f(v_n(x),\nabla v_n(x))dx
\end{equation}
for every $v_n\rightharpoonup_{W^{1,1}}u$ then
\begin{equation}
    sc^-(\overline{F_1})(u)=\int_{\Omega}f^{**}(u(x),\nabla u(x))dx\,.
\end{equation}
\end{proof}

\begin{remark}
    An interesting question regards which conditions guarantee the continuity of $f^{**}$ in both variables. A first result in this direction will be present in a forthcoming paper, in collaboration with Paulin Huguet.
\end{remark}

\begin{remark} In general the result of Theorem \ref{Theo fin} for non autonomous Lagrangians is false and it is possible to find controexamples in \cite{Bousquet2025}.
\end{remark}

\section*{Acknowledgment}
Firstly I want to thank my supervisor Giulia Treu for her important help in the writing of this paper. The results presented in this paper had also benefited a lot by the discussions with Pierre Bousquet and Paulin Huguet during my visit in Toulouse. This research has been funded by the GNAMPA project "Fenomeno di Lavrentiev, Bounded Slope Condition e regolarità per minimi di funzionali integrali con crescite non standard e lagrangiane non uniformemente convesse" (2024).


\begin{thebibliography}{11}
\bibitem{Alberti1994} G. Alberti, F. Serra Cassano, \textit{Non-occurrence of gap for one-dimensional autonomous functionals}, Calculus of variations, homogenization and continuum mechanics (Marseille, 1993), 1–17. Ser. Adv. Math. Appl. Sci., 18, 1994.
\bibitem{ABM}H. Attouch, G. Buttazzo, G. Michaille, \textit{Variational Analysis in Sobolev and BV Spaces}, MPS-SIAM Series on Optimization, 2006.
\bibitem{BallMizel1985}J.M. Ball, V.J. Mizel \textit{One-dimensional Variational Problems whose Minimizers do not Satisfy the Euler-Lagrange Equation}, 1985.
\bibitem{ButtazzoBelloni1995}M. Belloni, G. Buttazzo, \textit{A survey on old and recent results about the gap phenomenon in the calculus of variations}, Kluwer Academic Publishers Group, Dordrecht,1995.
\bibitem{Bousquet2025}M. Borowski, P. Bousquet, I. Chlebicka, B. Lledos, B. Miasojedow. \textit{Discarding
Lavrentiev’s Gap in Non-automonous and Non-Convex Variational Problems}, 2024 (to appear). hal-04814888. 
\bibitem{Bousquet2023}P. Bousquet, \textit{Nonoccurence of the Lavrentiev gap
for multidimensional autonomous problems}, 2023.
\bibitem{BMT2014}P. Bousquet, C. Mariconda, G. Treu, \textit{On the Lavrentiev phenomenon for multiple integral scalar variational problems}, Journal of Functional Analysis Volume 266, Issue 9, pp 5921-5954, 2014.
\bibitem{Bousquet2024}P. Bousquet, C. Mariconda, G. Treu, \textit{Non occurrence of the Lavrentiev gap for a class of
nonautonomous functionals}, Ann. Mat. Pura Appl. (4)   203, no. 5, pp 2275–2317, 2024.
\bibitem{Buttazzo1989}G. Buttazzo, \textit{Semicontinuity, relaxation and integral representation in the calculus of variations}, Longman Scientific \& Technical, 1989.
\bibitem{ButtazzoDalMaso1980} G. Buttazzo, G. Dal Maso, \textit{$\Gamma$-limits of integral functionals}, J. Analyse Math.   37, pp 145–185, 1980
\bibitem{ButtazoDalMaso1985} G. Buttazzo, G. Dal Maso, \textit{Integral representation and relaxation of local functionals}, Nonlinear Analysis: Theory, Methods \& Applications
Volume 9, Issue 6, pp 515-532, 1985.
\bibitem{ButtazzoDalMasoDeGiorgi1983}G. Buttazzo, G. Dal Maso, E. De Giorgi, \textit{On the lower semicontinuity of certain integral functionals}, Atti della Accademia Nazionale dei Lincei. Classe di Scienze Fisiche, Matematiche e Naturali. Rendiconti Lincei. Matematica e Applicazioni, 1983. 
\bibitem{ButtazzoLeaci1984}G. Buttazzo, A. Leaci, \textit{A Continuity Theorem for Operators from $W^{1,p}$ into $L^p$}, Journal of Functional Analysis 58, pp 216-224, 1984.
\bibitem{ButtazzoLeaci1985}G. Buttazzo, A. Leaci, \textit{Relaxation Results for a Class of Variational Integrals}, Journal of Functional Analysis 61, pp 360-377, 1985.
\bibitem{ButtazzoMizel1992} G. Buttazzo, V. J. Mizel, \textit{Interpretation of the Lavrentiev phenomenon by relaxation}, J. Funct. Anal. 110, no. 2, pp 434–460. 1992.
\bibitem{Cellina 1}A. Cellina, \textit{On minima of a functional of the gradient: necessary conditions}, Non Linear Analysis, 20(4), pp 337-341, 1993.
\bibitem{Cellina 2}A. Cellina, \textit{On minima of a functional of the gradient: sufficient conditions}, Non Linear Analysis, 20(4), pp 343-347, 1993.
\bibitem{Dac} B. Dacorogna,
\textit{Introduction to the calculus of variations},
  Imperial College Press, London, 2015.
\bibitem{DalMaso1988}G. Dal Maso, \textit{Relaxation of Autonomous Functionals with Discontinuous Integrands}, nn. Univ. Ferrara - Sez. VII - 8c. Mat. Vol. XXXIV, pp 21-47, 1988. 
\bibitem{DalMaso1993}G. Dal Maso, \textit{An Introduction to $\Gamma$-Convergence}, Birkhäuser, 1993.
\bibitem{ET}I. Ekeland, R. Témam, \textit{Convex Analysis and
Variational Problems}, 1999 (First edition 1976).
\bibitem{Friesecke1994}G. Friesecke, \textit{A necessary and sufficient condition for nonattainment and formation of microstructure almost everywhere in scalar variational problems}, Proceedings of the Royal Society of Edinburgh: Section A Mathematics, 124(3): pp 437-471, 1994.
\bibitem{Lavrentiev}M. Lavrentiev, \textit{Sur quelques problemes du calcul des variations}, Ann.
Matem. Pura Appl. 4, pp 7–28, 1926.
\bibitem{Mizel1988} A. C. Heinricher,; V. C: Mizel, \textit{The Lavrentiev Phenomenon for lnvariant Variational Problems},  Archive for rational mechanics and analysis, Vol.102(1), pp 57-93, 1988.
\bibitem{Manià1934}B. Manià, \textit{Sopra un esempio di Lavrentieff}, Boll. Un. Matem. Ital. 13, pp 147–153. 1934.
\bibitem{MS}P. Marcellini, C. Sbordone, \textit{Semicontinuity problems in the Calculus of Variations}, \textit{Non linear Analysis, Theory, Methods \& Applications} Val.4, No.2, pp 241-257, 1980.
\bibitem{Mariconda2023} C. Mariconda, \textit{Non-occurrence of gap for one-dimensional non-autonomous functionals}, Calc. Var. Partial Differential Equations   62, no. 2, Paper No. 55, pp 22, 2023.  
\bibitem{}A.M. Ribeiro, E. Zappale, \textit{Existence of minimizers for nonlevel convex supremal functionals} Siam J. Control Optim. Vol. 52, No. 5, pp 3341–3370, 2014.
\bibitem{Serrin 1961}J. Serrin, \textit{On the definition and properties of certain variational integrals}, Trans. Amer. Math. Soc.   101, pp 139–167, 1961
\bibitem{S}S. Stanis\l{}aw, \textit{Theory of the Integral}, english translation by L. C. Young, Hafner Publishing Company, 1964. 
\bibitem{Villa} S. Villa,\textit{Well Posedness of nonconvex integral functionals}, Siam J. Control Optim.
Vol. 43, No. 4, pp 1298–1312, 2005.
\bibitem{Zhykov1995}V.V. Zhikov, \textit{On Lavrentiev’s phenomenon} Russ. J. Math. Phys. 3(2), pp 249–269, 1995.
\end{thebibliography}
\end{document}